\documentclass[12pt]{amsart}
\usepackage{amsmath,amsthm,amssymb}
\usepackage{cite}

\usepackage{graphicx}
\usepackage{tikz}
\usetikzlibrary{calc}
\usetikzlibrary{positioning}
\usepackage{subfigure}
\usepackage{geometry} 
\newcommand{\Var} {\mathrm {Var}}
\newcommand{\Prob} {\mathbb {P}}
\newcommand{\prob}[1]{\Prob\left(#1\right)}

\newcommand{\nc}{\mathcal{N}}
\newcommand{\tc}{\mathcal{T}}
\newcommand{\htc}{\hat{\tc}}

\newcommand{\E}{\mathbb{E}}
\newcommand{\I} {\mathbb{I}}

\newcommand\dto{\overset{\mathrm{d}}{\to}}
\newcommand{\po}{{^{(M)}}}

\newtheorem{theorem}{Theorem}
\newtheorem{lemma}[theorem]{Lemma}

\newtheorem{proposition}[theorem]{Proposition}

\newtheorem*{conjecture*}{Conjecture}
\newtheorem*{proposition*}{Proposition}
\theoremstyle{remark}
\newtheorem{remark}[theorem]{Remark}
\newtheorem*{remark*}{Remark}

\begin{document}
\author{Dimbinaina Ralaivaosaona}
\address{Dimbinaina Ralaivaosaona and Stephan Wagner \\ Department of Mathematical Sciences \\ Stellenbosch University \\ Private Bag X1 \\ Matieland 7602, South Africa.}
\email{\{naina,swagner\}@sun.ac.za}
\author{Matas \v{S}ileikis}
\address{Matas \v{S}ileikis \\ The Czech Academy of Sciences \\ Institute of Computer Science \\ Pod Vod\'{a}renskou v\v{e}\v{z}\'{\i} 2 \\ 182 07 Prague, Czech Republic.}
\email{matas.sileikis@gmail.com}
\author{Stephan Wagner}
\title{A central limit theorem for almost local additive tree functionals}
\date{\today}
\thanks{The first author was partially supported by the Division for Research Development (DRD) of Stellenbosch
University. The second author was supported by the Czech Science Foundation, grant number GJ16-07822Y, with institutional support RVO:67985807. The third author was supported by the National Research Foundation of South Africa, grant 96236. An extended abstract of this paper appeared in the Proceedings of the 29th International Conference on Probabilistic, Combinatorial and Asymptotic Methods for the Analysis of Algorithms, AofA 2018, see \cite{RSW18}}

\begin{abstract}
An additive functional of a rooted tree is a functional that can be calculated recursively as the sum of the values of the functional over the branches, plus a certain toll function. Janson recently proved a central limit theorem for additive functionals of conditioned Galton-Watson trees under the assumption that the toll function is local, i.e.~only depends on a fixed neighbourhood of the root. We extend his result to functionals that are ``almost local'' in a certain sense, thus covering a wider range of functionals. The notion of almost local functional intuitively means that the toll function can be approximated well by considering only a neighbourhood of the root. Our main result is illustrated by several explicit examples including natural graph theoretic parameters such as  the number of independent sets, the number of matchings, and the number of dominating sets. We also cover a functional stemming from a tree reduction process that was studied by Hackl, Heuberger, Kropf, and Prodinger.\end{abstract}

\maketitle

\section{Introduction}

A functional $F$ that associates a value $F(T)$ with every rooted tree is said to be \emph{additive} if it satisfies a recursion of the form
\begin{equation}\label{eq:additive}
F(T) = \sum_{i=1}^k F(T_i) + f(T),
\end{equation}
where $T_1,T_2,\ldots,T_k$ are the branches of $T$ and $f$ is a so-called ``toll function'', another function that assigns a value to every rooted tree. If $T$ only consists of the root (so that $k=0$), we interpret the empty sum as $0$ and set $F(T) = f(T)$. Of course, every functional $F$ is additive in this sense (for a suitable choice of $f$), so the usefulness of the concept depends on what is known about the toll function $f$.

An important special case of an additive functional is the number of occurrences of a prescribed ``fringe subtree''. A fringe subtree is an induced subtree of a rooted tree that consists of one of the nodes and all its descendants. Now fix a rooted tree $S$. We say that $S$ occurs on the fringe of $T$ if there is a fringe subtree of $T$ that is isomorphic to $S$ (when we consider ordered trees, where the order of branches matters, ``isomorphic'' is to be understood in the ordered sense as well). The number of occurrences of $S$ as a fringe subtree in $T$ (i.e., the number of nodes $v$ of $T$ for which the fringe subtree rooted at $v$ is isomorphic to $S$) is an additive functional, which we shall denote by $F_S(T)$. Indeed, one has
$$F_S(T) = \sum_{i=1}^k F(T_i) + f_S(T),$$
where
$$f_S(T) = \begin{cases} 1 & S \text{ is isomorphic to } T, \\ 0 & \text{otherwise.} \end{cases}$$
This is because an occurrence of $S$ in $T$ is either an occurrence in one of the branches, or comprises the entire tree $T$. Every additive functional can be expressed as a linear combination of these elementary functionals: it is easy to see (for example~by induction) that a functional satisfying~\eqref{eq:additive} can be expressed as
$$F(T) = \sum_S f(S) F_S(T).$$
Functionals of the form $F_S$ are known to be asymptotically normally distributed in different classes of trees, notably simply generated trees/Galton-Watson trees \cite{J16,W15}, which will also be the topic of this paper, and classes of increasing trees \cite{HJS17,RW16}. In view of this and several other important examples of additive functionals that satisfy a central limit theorem, general schemes have been devised that yield a central limit theorem under different technical assumptions. This includes work on simply generated trees/Galton-Watson trees \cite{J16,W15} (labelled trees, plane trees and $d$-ary trees are well-known special cases) as well as  P\'olya trees \cite{W15} and increasing trees \cite{RW16,W15} (specifically recursive trees, $d$-ary increasing trees and generalised plane-oriented recursive trees). It is worth mentioning, however, that there are also many instances of additive functionals that are not normally distributed in the limit, since the toll functions can be quite arbitrary. A well-known example is the case of the path length, i.e.~ the sum of the distances of all nodes to the root. It satisfies~\eqref{eq:additive} with toll function
$$f(T) = |T|-1,$$
and, when suitably normalised, its limiting distribution for simply generated trees is the Airy distribution (see \cite{T91}).

Previous results \cite{HJS17,J16,RW16,W15}, while giving rather general conditions on the toll function that imply normality, are unfortunately still insufficient to cover all possible examples one might be interested in. This paper is essentially an extension of Janson's work \cite{J16} on local functionals. By weakening the conditions he makes on the toll functions, we arrive at a new general central limit theorem that can be applied to a variety of examples that were not previously covered. Several such examples are presented in detail in this paper including natural graph theoretical parameters and an open problem from a paper of Hackl, Heuberger, Kropf and Prodinger \cite{HHKP17} on tree reductions.

A \emph{local} functional (as considered in Janson's paper \cite{J16}) is a functional for which the value of the toll function can be determined from the knowledge of a fixed neighbourhood of the root. A typical example is the number of nodes with a given outdegree $r$, where the corresponding toll function is completely determined by the root degree: its value is $1$ if the root degree is $r$, and $0$ otherwise. We relax this condition somewhat (to what we call ``almost local functionals'') in our main theorem. Intuitively speaking, functionals that satisfy our conditions have toll functions that can be approximated well from knowledge of a neighbourhood of the root, with the approximation getting better the wider the neighbourhood is chosen.

The model of random trees that we consider here are \emph{conditioned Galton-Watson trees}: these are determined by an offspring distribution $\xi$, which we will assume to satisfy $\E \xi = 1$. We also assume that $\mathrm{Var} \xi$ is finite and nonzero (to avoid a degenerate case). The Galton-Watson process starts from a single node, the root. At time $t$, all nodes at level/depth $t$ (distance $t$ from the root) generate a number of children according to the offspring distribution $\xi$. The numbers of children of different nodes on the same level are mutually independent. The outcome of this process, which ends when all nodes at level $t$ generate $0$ children, is a random tree $\tc$ (almost surely finite). By conditioning the process to ``die out'' when the total number of nodes is $n$ (of course, we only consider $n$ for which such an event occurs with nonzero probability) we obtain a conditioned Galton-Watson tree, which will be denoted by $\tc_n$.

Conditioned Galton-Watson trees are known to be essentially equivalent to so-called \emph{simply generated trees} \cite[Section 3.1.4]{D09}. Classical examples include rooted labelled trees (corresponding to a Poisson distribution for $\xi$), plane trees (corresponding to a geometric distribution for $\xi$) and binary trees (with a distribution whose support is $\{0,2\}$).

We conclude the introduction with some more notation: for a tree $T$, we let $T^{(M)}$ be its restriction to the first $M$ levels, i.e.~all nodes whose distance to the root is at most $M$. A local functional as defined above is thus a functional for which the value of $f(T)$ is determined by $T^{(M)}$ for some fixed $M$ (the ``cut-off''). 
The conditioned Galton-Watson tree $\tc_n$ is known to converge in the local topology induced by these restrictions to the (infinite) \emph{size-biased} Galton-Watson tree $\htc$ as defined by Kesten \cite{K86}, see also \cite{J12}: one has
$$\Prob(\htc^{(M)} = T) = w_M(T) \Prob(\tc^{(M)} = T)$$
for all trees $T$, where $w_M(T)$ is the number of nodes of depth $M$ in $T$. 

For a rooted tree $T$ (possibly infinite), we let $\deg(T)$ denote the degree of the root of $T$. Finally, it will be convenient for us to use the Vinogradov notation $\ll$ interchangeably with the $O$-notation, i.e.~$f(n) \ll g(n)$ and $f(n) = O(g(n))$ both mean that $|f(n)| \leq K g(n)$ for a fixed positive constant $K$ and all sufficiently large $n$.

\section{The general theorem}

Let us now formulate our main result, which is a central limit theorem for additive functionals under suitable technical conditions on the toll function $f$.

\begin{theorem}\label{thm:normality}
	Let $\tc_n$ be a conditioned Galton-Watson tree of order $n$ with offspring distribution $\xi$, where $\xi$ satisfies $\E \xi = 1$ and $0<\sigma^2:=\mathrm{Var} \xi <\infty$. Assume further that $\E \xi^{2\alpha+1} < \infty$ for some integer $\alpha\geq 0$. Consider a functional $f$ of finite rooted ordered trees with the property that  
	\begin{equation}\label{eq:thm1-1}
	f(T) = O(\deg (T)^{\alpha}).
	\end{equation}
	Furthermore, assume that there exists a sequence $(p_M)_{M\geq 1}$ of positive numbers with  $p_M\to 0$, as $M\to\infty$, such that
	
	\begin{itemize}
		\item for every $M \in \{1, 2, \dots\}$, 
		\begin{equation}\label{eq:thm1-2}
		\, \E \left|  f(\hat \tc^{(M)})-\E\left(f(\hat \tc^{(N)})\,  | \, \hat \tc^{(M)}\right)\right|\leq p_M
		\end{equation}
		for all $N$ and $M$ with $N\geq M$,
		\item there is a sequence of positive integers $(M_n)_{n\geq 1}$ such that for large enough $n$,
		\begin{equation}\label{eq:thm1-3}
		\E \left|f(\tc_n)-f\left(\tc_n ^{(M_n)}\right)\right|\leq p_{M_n}.
		\end{equation}
	\end{itemize}
	If $a_n:=n^{-1/2}(n^{\max\{\alpha, 1\}}p_{M_n}+M_n^2)$ satisfies
	\begin{equation}\label{eq:thm1-4}
	\lim_{n\to\infty} a_n= 0, \, \text{ and } \, \sum_{n=1}^{\infty}\frac{a_n}{n}<\infty, 
	\end{equation}
	then
	\begin{equation}\label{eq:normality}
	\frac{F(\tc_n) - n\mu}{\sqrt n} \dto \nc(0,\gamma^2)
	\end{equation}
	where $\mu=\E f(\tc)$, and $0 \le \gamma < \infty$.
\end{theorem}
In \eqref{eq:normality}, the numerator $F(\tc_n) - n\mu$ can be replace by $F(\tc_n) - \E F(\tc_n)$ as we will see in Proposition~\ref{prop:mean} that 
\begin{equation}\label{eq:meanF}
\E F(\tc_n) = n \mu +o(\sqrt{n}), \, \text{ as } n\to \infty.
\end{equation}
the existence of the expectation of $\mu=\E f(\tc)$ is guaranteed by \eqref{eq:thm1-1}.

\begin{remark}\label{rem:1}
	The proof of Theorem~\ref{thm:normality} is a generalisation of Janson's proof of his theorem for bounded and local functionals in \cite{J16}. The boundedness condition is now replaced by \eqref{eq:thm1-1} assuming finiteness of higher moments of the offspring distribution $\xi$. However, the main difficulty to overcome is the fact that our toll function is no longer local. To give a simple example, an essential part of the proof is to give a meaning to the ``expectation" $\E f(\htc)$. The functional $f$ does not need to be defined on infinite trees. When $f$ is local with a cut-off $M$, then  $f(\htc)=f(\htc\po)$ by definition. So, $\E f(\htc)$  is simply defined to be $\E f(\htc\po)$. In our case, where $f$ is not necessarily local, we can define 
	\begin{equation}\label{eq:rmk1}
	\E f(\htc):=\lim_{M\to\infty} \E f(\htc^{(M)}),
	\end{equation}
	which may not exist in general. However, if $f$ satisfies \eqref{eq:thm1-2}, then we can show that $\E f(\htc)$ exists. Indeed, 
	\begin{align*}
		|\E  f(\htc^{(M)}) -\E f(\htc^{(N)}) | & = \left|\E \left( f(\hat \tc^{(M)})-\E\left(f(\hat \tc^{(N)})\,  | \, \hat \tc^{(M)}\right)\right)\right|\\
		& \leq \E \left|  f(\hat \tc^{(M)})-\E\left(f(\hat \tc^{(N)})\,  | \, \hat \tc^{(M)}\right)\right| \leq p_M,
	\end{align*}
	which tends to zero as $M\to\infty$, uniformly for $N\geq M$. In other words,  $(\E f(\hat \tc^{(M)}))_{M\geq 1}$ is a Cauchy sequence, so the limit \eqref{eq:rmk1} exists.
\end{remark}

\section{Auxiliary results}
In this section, we give some useful results that we will need in the proof of our main theorem.   Throughout the rest of the paper, the offspring distribution $\xi$ is assumed to satisfy  $\E \xi=1$,  $0<\sigma^2:=\mathrm{Var} \xi<\infty$, and $\E \xi^{2\alpha+1}<\infty$ for some fixed integer $\alpha\geq 0$ as in Theorem~\ref{thm:normality}. The distribution of the number of nodes at level $k$, $w_k$, for the three random trees $\tc$, $\htc$, and $\tc_n$ will play an important role in our proof.  This parameter has been studied in \cite{J06}, and in particular, the following results were proved there, see \cite[Theorem 1.13, Lemma 2.2, and Lemma 2.3]{J06} (note that $\tc_\infty$ is used there for $\htc$).
\begin{lemma}\label{lem:Aux1}
 For every positive integer  $r\leq \max \{2\alpha, \, 1\}$, we have
\begin{equation}\label{eq:prem1}
\E\left(w_k(\tc)^{r}\right) =O (k^{r-1}), \,  \, \E(w_k(\htc)^{r}) = O(k^{r}), \text{ and }\,  \E \left(w_k(\tc_n)^r\right) =O (k^r),
\end{equation}
where the constants in the $O$-terms depend on the offspring distribution $\xi$ only.
\end{lemma}
For a rooted tree $T$, we know that $|T\po|=\sum_{k=0}^{M}w_k(T)$. Hence, we can immediately deduce from this lemma, with $r=1$,  that 
\begin{equation}\label{eq:prem2}
\E |\tc\po|=O(M),\, \E |\htc\po|  =O (M^2), \, \text{ and } \E |\tc_n\po|  =O (M^2).
\end{equation}
In fact, it can be shown that $\E |\tc\po|=M+1$. 
We are also going to make extensive use of the higher moments of the root degree. By definition, the distribution of $\deg(\tc)$ is $\xi$, so we know the higher moments of $\deg(\tc)$. On the other hand, note that $\deg(T)=w_1(T)$. So, as  particular cases of the estimates in \eqref{eq:prem1}, we have 
\begin{equation}\label{eq:prem3}
\E (\deg(\htc)^r) <\infty \, \text{ and }  \E \left(\deg(\tc_n)^r\right) =O(1),
\end{equation}
for every positive integer $r\leq \max \{2\alpha, \, 1\}$, where the implied constant in the second estimate is independent of $n$. 

Its well known that $\tc_n$ converges locally to the infinite random tree $\htc$ in the sense that 
\[
\prob{\tc_n\po=T}\to  \prob{\htc\po=T}, \, \text{ as }\, n\to\infty,
\] 
for any fixed integer $M\geq 0$ and a fixed tree $T$. Janson obtained a bound on the rate of convergence of this estimate, in the proof of  \cite[Lemma 5.9]{J16} (see (5.42) there). His result can be formulated as follows: For any tree $T$ with $|T| \le n/2$, and any $M\geq 0$ we have
	\begin{equation}\label{eq:contiguity}
	\prob{\tc_n\po = T} = \prob{\htc\po = T} \left( 1 + O \left( \frac{|T|}{n^{1/2}} \right) \right),
	\end{equation}
	where the constant in the $O$ notation is independent of $T$ and $M.$  Here $T$ and $M$ may depend on $n$, and this is crucial for our purposes. As a consequence of this result we can bound the difference  between the two expectations $\E f(\tc_n\po)$ and $\E f(\htc\po)$ explicitly in terms of $M$, where $M$ is also allowed to depend on $M.$     

\begin{lemma}\label{lem:aux_exp}
	If $f$ satisfies \eqref{eq:thm1-1}, then we have
	\begin{equation}\label{eq:Mean-Last}
	|\E f(\tc_n\po) - \E f(\htc\po)| = O\left(n^{-1/2} M^2 \E (\deg(\htc)^{\alpha+1})+n^{-1}M^2\E (\deg(\tc_n)^{\alpha+1})\right),
	\end{equation}
	where the constant in the $O$ notation is independent of $n$ and $M.$ 
\end{lemma}

\begin{proof}
We have 
	\begin{align*}\label{eq:truncatedExps}
		|\E f(\tc_n\po) -  \E  f(\htc  \po)|  
		 = &   \left|\sum_T f(T)\prob{\tc_n\po = T} - \sum_T f(T) \prob{\htc\po = T}\right|\\
		  \le & \sum_{|T| \le n/2}|f(T)|\left|\prob{\tc_n\po = T} - \prob{\htc\po = T}\right| + \\
		&\sum_{|T| > n/2} |f(T)|\left( \prob{\tc_n\po = T} + \prob{\htc\po = T} \right) \\
	\end{align*}
	
We begin by estimating the sum over $|T|\le n/2$. Using \eqref{eq:contiguity} and the bound \eqref{eq:thm1-1} on $f(T)$, we obtain 
\[
\sum_{|T| \le n/2}|f(T)|\left|\prob{\tc_n\po = T} - \prob{\htc\po = T}\right|\ll  \sum_T  \prob{\htc\po = T}\frac{\deg(T)^{\alpha}|T|}{n^{1/2}}.
\] 	
Now, the right-hand side can be bounded as follows: 
	\begin{align*}
		\sum_T  \prob{\htc\po = T}\frac{\deg(T)^{\alpha}|T|}{n^{1/2}}  
		& = n^{-1/2}\E (\deg(\htc\po)^{\alpha}|\htc\po|) \\
		& = n^{-1/2}\E  \left(\deg(\htc\po)^{\alpha}\, \E \Big(|\htc\po| \, \Big| \, \deg(\htc\po)\Big)\right).
	\end{align*}
	Conditioning on  $\deg(\htc\po)$ (which is the same as $\deg(\htc)$ for $M\geq1$),  $\htc$ consists of a root, a copy of $\htc$ and  $\deg(\htc)-1$ independent copies of $\tc$. Thus, by the estimates in \eqref{eq:prem2}, we have
	\begin{equation}\label{eq:cond-m}
	\E \Big(|\htc\po| \, \Big| \, \deg(\htc\po)\Big) \ll M^2 + M \deg(\htc\po) \ll M^2\deg(\htc\po).	
	\end{equation}
	Therefore,
	$$
	\E  \left(\deg(\htc\po)^{\alpha}\, \E \Big(|\htc\po| \, \Big| \, \deg(\htc\po)\Big)\right) \ll M^2 \E (\deg(\htc\po)^{\alpha+1}),
	$$
	which yields
	\begin{equation}\label{eq:part1}
	\sum_T  \prob{\htc\po = T}\frac{\deg(T)^{\alpha}\, |T|}{n^{1/2}} \ll n^{-1/2}M^2\E (\deg(\htc\po)^{\alpha+1}).
	\end{equation}
	
	Next we estimate the sum over $|T|> n/2$, which we split further as follows:
\begin{align*}
\sum_{|T| > n/2} |f(T)|& \left( \prob{\tc_n\po = T} + \prob{\htc\po = T} \right) \\
& \leq \sum_{|T| > n/2} \prob{\htc\po = T} \deg(T)^{\alpha}+\sum_{|T| > n/2}  \prob{\tc_n\po = T} \deg(T)^{\alpha}.
\end{align*}
We have	
\begin{align*}
		\sum_{|T| > n/2}& \prob{\htc\po = T}   \deg(T)^{\alpha} \\ 
		& = \sum_{k\geq 1} k^{\alpha}\, \prob{|\htc\po |> n/2 \, \text{ and } \, \deg(\htc)=k}\\
		& = \sum_{k\geq 1} k^{\alpha}\, \prob{\deg(\htc)=k} \prob{|\htc\po |> n/2 \, \Big| \, \deg(\htc)=k}.
	\end{align*}
	Markov's inequality yields
	$$
	\prob{|\htc\po |> n/2 \, \Big| \, \deg(\htc)=k} \leq \frac{2\E (|\htc\po | \,| \, \deg(\htc)=k )}{n} \ll \frac{kM^2}{n},
	$$
	where the last estimate follows from \eqref{eq:cond-m}. Thus, 
	\begin{equation}\label{eq:part2}
	\sum_{|T| > n/2} \prob{\htc\po = T}  \deg(T)^{\alpha}\, \ll n^{-1}M^2 \E (\deg(\htc)^{\alpha+1}).
	\end{equation}
	Finally, for the last term, we proceed in a similar fashion:
	\begin{align*}
		\sum_{|T| > n/2}&  \prob{\tc_n\po = T}  \deg(T)^{\alpha}\, \\ 
		& \leq \sum_{k\geq 1} k^{\alpha}\, \prob{|\tc_n\po |> n/2 \, \text{ and } \, \deg(\tc_n)=k}\\
		& \leq \sum_{k\geq 1} k^{\alpha}\, \prob{\deg(\tc_n)=k} \prob{|\tc_n\po |> n/2 \, \Big| \, \deg(\tc_n)=k}.
	\end{align*}
	If $\tc_{n,1}, \tc_{n,2}, \dots, \tc_{n,k}$ are the branches of $\tc_n$, given that $\deg(\tc_n)=k$, then, conditioning on their sizes $n_1, n_2, \dots, n_k$, they are $k$ independent conditioned Galton-Watson trees $\tc_{n_1}, \tc_{n_2}, \dots, \tc_{n_k}$. On the other hand, we have 
	$$
	|\tc_n\po|=1+\sum_{i=1}^{k}|\tc_{n,i}^{(M-1)}|.
	$$
	Thus,
	\begin{align*}
		\E\left(|\tc_n^{(M)}| \, | \, \deg(\tc_n)=k\right)
		& =\E \left(\E\left(|\tc_n^{(M)}| \, \Big| \, n_1,  n_2, \cdots, n_k\right)\right)\\
		& =1+ \sum_{i=1}^{k} \E \left(\E\left(|\tc_{n_i}^{(M-1)}| \, \Big| \, n_1,  n_2, \cdots, n_k\right)\right)   \ll kM^2,
	\end{align*}
	which again follows from the last estimate in \eqref{eq:prem2}. Now, Markov's inequality yields
	$$
	\prob{|\tc_n\po |> n/2 \, \Big| \, \deg(\tc_n)=k} \ll n^{-1}kM^2.
	$$
	Therefore, making use of \eqref{eq:prem3} once again, we have
	\begin{multline}\label{eq:part3}
		\sum_{|T| > n/2} \prob{\tc_n\po = T}  \deg(T)^{\alpha}\,\\ \ll  n^{-1}M^2\sum_{k\geq 1}k^{\alpha+1}\,\prob{\deg(\tc_n)=k} 
		=n^{-1}M^2\E (\deg(\tc_n)^{\alpha+1}).
	\end{multline}
	Combining the estimates \eqref{eq:part1},  \eqref{eq:part2}, and \eqref{eq:part3}, we finally arrive at the estimate 
	\begin{equation}\label{eq:middle}
	|\E f(\tc_n\po) - \E f(\htc\po)| \ll n^{-1/2}M^2\E (\deg(\htc)^{\alpha+1})+n^{-1}M^2\E (\deg(\tc_n)^{\alpha+1}),
	\end{equation}
	which complete the proof of the lemma.
\end{proof}

The following lemma will be useful in the estimate of the variance in the next section. First, we start with some operators on functionals. For any  toll function $f$  of an additive functional $F$, we denote by $f^{(0)}$ the centred toll function which is defined by
\[
f^{(0)}(T):=f(T)-\E f(\tc_{|T|}),
\]
and let $F^{(0)}$ be the additive functional associated with $f^{(0)}.$ Furthermore,
for any  a subset $\mathcal{S}$ of $\mathbb{N}$,  let $f_{\mathcal{S}}$ be the functional defined by 
\[
f_{\mathcal{S}}(T)=
\begin{cases}
f(T) \, & \text{ if } \, |T|\in \mathcal{S},\\
0 \, & \text{ otherwise,}
\end{cases}
\]
and we denote by $F_{\mathcal{S}}$ the additive functional whose toll function is $f_{\mathcal{S}}.$


\begin{lemma}\label{lem:Aux2} Assume that $f$ satisfies the conditions of Theorem~\ref{thm:normality} and let $(p_M)_{M\geq 1}$ and $(M_n)_{n\geq 1}$ be the corresponding sequences. Then for any subset $\mathcal{S}$  of $\mathbb{N},$ we have 
	\begin{equation}\label{eq:fF0}
	\E (f_{\mathcal{S}}^{(0)}(\tc_n)F_{\mathcal{S}}^{(0)}(\tc_n)) \ll n^{\max\{\alpha,\,1\}} \, p_{M_n}+ \E(\deg(\tc_{n})^{2\alpha})+M_n^2\, \E(\deg(\tc_{n})^{\alpha+1}).
	\end{equation}
\end{lemma}
\begin{proof} Since the left side of \eqref{eq:fF0} is zero for $n\notin \mathcal{S},$ we may assume without loss of generality that $n\in \mathcal{S}.$
	We decompose $F_{\mathcal{S}}^{(0)}(\tc_n)$ according to the depth $d(v)$ of the nodes:
	\begin{equation}
	F_{\mathcal{S}}^{(0)}(\tc_n) = \sum_{v \in \tc_n} f_{\mathcal{S}}^{(0)}(\tc_{n,v}) = \sum_{d(v) < M} f_{\mathcal{S}}^{(0)}(\tc_{n,v}) + \sum_{d(v) \ge M} f_{\mathcal{S}}^{(0)}(\tc_{n,v}) =: S_1 + S_2,
	\end{equation}
	where $\tc_{n,v}$ denotes the fringe subtree of $\tc_n$ rooted at $v.$
	Notice that $f_{\mathcal{S}}^{(0)}$ might not necessarily satisfy all conditions of Theorem \ref{thm:normality}. However,  \eqref{eq:thm1-1} is satisfied by $f_{\mathcal{S}}^{(0)}$. Hence, we have
	\begin{align*}
		\E |f_{\mathcal{S}}^{(0)}(\tc_n)S_1| & \ll \E \Big(\deg(\tc_n)^{\alpha}\sum_{d(v) < M} \deg(\tc_{n,v})^{\alpha}\Big)\\
		& = \E \Big(\deg(\tc_n)^{\alpha}\,\E\Big(\sum_{d(v) < M} \deg(\tc_{n,v})^{\alpha}\, \Big| \, \deg(\tc_n)\Big) \Big).
	\end{align*}
	Next, for any positive integer $m\leq M$, we have 
	$$
	\E\Big(\sum_{d(v) < m} \deg(\tc_{n,v})^{\alpha}\, \Big| \, \tc_n^{(m-1)}\Big)=\sum_{d(v)<m-1}\deg(\tc_{n,v})^{\alpha}+O\left( w_{m-1}(\tc_n)\right).
	$$	
	This is because the $w_{m-1}(\tc_n)$ fringe subtrees with roots at level $m-1$, conditioned on their sizes, are conditioned Galton-Watson trees and thus by \eqref{eq:prem3} the moments of the root degrees are $O(1)$. Taking the expectation conditioned on $\deg(\tc_n)$, again by the same argument, and by the estimate $\E w_{m-1}(\tc_n)=O(m)$ as in \eqref{eq:prem1}, we have
	$$
	\E\Big(\sum_{d(v) < m} \deg(\tc_{n,v})^{\alpha}\, \Big| \, \deg(\tc_n)\Big)=  \E\Big(\sum_{d(v) < m-1} \deg(\tc_{n,v})^{\alpha}\, \Big| \, \deg(\tc_n)\Big)+O(m\deg(\tc_n)).
	$$
	Thus, iterating from $M$, we obtain 
	$$
	\E\Big(\sum_{d(v) < M} \deg(\tc_{n,v})^{\alpha}\, \Big| \, \deg(\tc_n)\Big)\ll  \deg(\tc_{n})^{\alpha} + M^2\deg(\tc_n).
	$$	
	Therefore,
	\begin{equation}\label{eq:fs1}
	\E |f_{\mathcal{S}}^{(0)}(\tc_n)S_1|  \ll \E(\deg(\tc_{n})^{2\alpha})+M^2\E(\deg(\tc_{n})^{\alpha+1}).
	\end{equation}

	\medskip
	
	Now for the contribution from $S_2$, note first that 
\begin{equation}\label{eq:S2}
S_2=\sum_{d(v)=M}F_{\mathcal{S}}^{(0)}(\tc_{n,v}).
\end{equation}
We condition on $\tc_n\po$ and the sizes of the fringe subtrees $\tc_{n,v_i}$, $i =1, \dots, w_M(\tc_k)$, induced by nodes at level $M$. Conditionally, each $\tc_{n,v_i}$ is distributed as $\tc_{n_i}$,  where $n_i=|\tc_{n,v_i}|$. From the definition of $f_{\mathcal{S}}^{(0)}$, we know that  $\E f_{\mathcal{S}}^{(0)}(\tc_m) = 0$ for every $m\geq 1.$ Noting that 
\[
F_{\mathcal{S}}^{(0)}(\tc_m) =\sum_{k=1}^{m-1} F_{\{k\}}(\tc_m)\I_{\{k\in\mathcal{S} \}},
\]
where $\I_{\{k\in\mathcal{S}\}}$ denotes the indicator function of the set $\mathcal{S}$. The sum is only over $k<m$ since, trivially, $F_{\{k\}}(T)=0$ for $|T|<k.$ It follows (see \cite[(6.25)]{J16}) that $\E F_{\mathcal{S}}^{(0)}(\tc_m) = 0$ for every $m\geq 1$ and therefore, by \eqref{eq:S2} and the law of total expectation, we also have 
	\begin{equation}\label{eq:condS2}
	\E \left(S_2 \, | \, \tc_n\po \right) = 0. 
	\end{equation}
	Let us define $\tilde f_{\mathcal{S},M}(\tc_n):=\E(f_{\mathcal{S}}(\tc_n)\, | \, \tc_n\po)$ and similarly $\tilde f^{(0)}_{\mathcal{S},M}(\tc_n):=\E(f_{\mathcal{S}}^{(0)}(\tc_n)\, | \, \tc_n\po)$.	Then
	$$
	\E (\tilde f^{(0)}_{\mathcal{S},M}(\tc_n) S_2)=\E \Big( \E\left(\tilde f^{(0)}_{\mathcal{S},M}(\tc_n) S_2 \, | \, \tc_n\po\right) \Big) =\E\left( \tilde f^{(0)}_{\mathcal{S},M}(\tc_n)\E\left( S_2 \, | \, \tc_n\po\right)\right)=0.
	$$	
	Hence,	
	$$
	|\E (f_{\mathcal{S}}^{(0)}(\tc_n)S_2)| = |\E(S_2 (f_{\mathcal{S}}^{(0)}(\tc_n)- \tilde f^{(0)}_{\mathcal{S},M}(\tc_n)))|\leq \max |S_2|\, \E |f_{\mathcal{S}}^{(0)}(\tc_n)- \tilde f^{(0)}_{\mathcal{S},M}(\tc_n)|.
	$$
By definition, we know that $f_{\mathcal{S}}(\tc_n)=f_{\mathcal{S}}^{(0)}(\tc_n)+\E f_{\mathcal{S}}(\tc_n)$ and we can also verify that $\tilde f_{\mathcal{S},M}(\tc_n)=\tilde f^{(0)}_{\mathcal{S},M}(\tc_n)+\E f_{\mathcal{S}}(\tc_n).$ Therefore, since we assumed that $n\in \mathcal{S}$, we have  	
\[
\E |f_{\mathcal{S}}^{(0)}(\tc_n)- \tilde f^{(0)}_{\mathcal{S},M}(\tc_n)| = \E |f_{\mathcal{S}}(\tc_n)- \tilde f_{\mathcal{S},M}(\tc_n)| =\E |f(\tc_n)- \E(f(\tc_n)\, | \, \tc_n\po)|.
\]
We can  use  \eqref{eq:thm1-2} and \eqref{eq:thm1-3} to estimate the right-hand side of the above equation. For the rest of the proof, we choose $M=M_n$ (as defined in Theorem \ref{thm:normality}). We have
	\begin{align*}
		|f(\tc_n)- \E(f(\tc_n)\, | \, \tc_n\po)|
		& \leq |f(\tc_n)-f(\tc_n\po)|+|f(\tc_n\po)-\E(f(\tc_n)\, | \, \tc_n\po)|\\
		& = |f(\tc_n)-f(\tc_n\po)|+\left|\E \left(f(\tc_n\po)-f(\tc_n)\, | \, \tc_n\po\right)\right|\\
		& \leq |f(\tc_n)-f(\tc_n\po)|+\E\left(|f(\tc_n\po)-f(\tc_n)|\, |\, \tc_n\po\right).
	\end{align*}
	Taking the expectation again, and using our condition \eqref{eq:thm1-3} (with $M=M_n$), we obtain
	$$
	\E |f(\tc_n)- \E(f(\tc_n)\, | \, \tc_n\po)|\leq 2 p_M.
	$$
 On the other hand, we have 
	$$
	|S_2| \leq \sum_{d(v) \ge M} |f^{(0)}_{\mathcal{S}}(\tc_{n,v})|\ll \sum_{v\in \tc_n} \deg(\tc_{n,v})^{\alpha}. 
	$$
	Since $\alpha$ is a nonnegative integer, the last term is bounded above by $(\sum_{v\in \tc_n} \deg(\tc_{n,v}))^{\alpha}$ (which is equal to $(n-1)^{\alpha}$) except for $\alpha=0$. Hence, we get 
	$$
	\max |S_2| \leq n^{\max\{\alpha, 1\}}.
	$$
	Therefore, putting everything together, we have 
	$$
	\E (f^{(0)}_\mathcal{S}(\tc_n)F^{(0)}_\mathcal{S}(\tc_n)) \ll n^{\max\{\alpha, 1\}}p_M+\E(\deg(\tc_{n})^{2\alpha})+M^2\E(\deg(\tc_{n})^{\alpha+1}),
	$$
	as claimed.
\end{proof}


\section{Mean and variance}	
We first look at the expectation $\E f(\tc_n)$. As it is also the case in \cite{J16}, one of the key observations in the proof of Theorem~\ref{thm:normality} is the fact that $\E f(\tc_n)$ is asymptotically equal to $\E f(\htc)$ (which is finite, cf. Remark~\ref{rem:1}) with an explicit bound on the error term. This is made precise in the next lemma.

\begin{lemma}\label{lem:expectation}
	If $f$ satisfies the conditions of Theorem~\ref{thm:normality}, then
	\begin{equation}
	\E f(\tc_n) = \E f(\htc) + O(p_{M_n} + n^{-1/2} \, M_n^2 ).
	\end{equation}
\end{lemma}
\begin{proof}
	We let $M_n$ be defined as in Theorem~\ref{thm:normality}, but write $M = M_n$ for easier reading. Notice first that 
	\begin{multline}\label{eq:diff}
		|\E f(\tc_n) - \E f(\htc)| \\
		\le |\E f(\tc_n) - \E f(\tc_n\po)|  + |\E f(\hat \tc\po) - \E f(\htc)|+ |\E f(\tc_n\po) - \E f(\htc\po)|.
	\end{multline}
	The first term on the right side is at most $p_M$ by assumption \eqref{eq:thm1-3}. The second term is also bounded above by $p_M$ in view of \eqref{eq:thm1-2}, using the same argument as in Remark \ref{rem:1}:  we have
	\begin{align*}
		|\E f(\htc^{(N)}) - \E f(\htc\po)|  
		&=  \left|\E \left( f(\htc\po)-\E\left(f(\htc^{(N)})|\htc^{(M)}\right)\right)\right|\\
		&\le \E \left|  f(\htc\po)-\E\left(f(\htc^{(N)})|\htc^{(M)}\right)\right|\leq p_M,
	\end{align*}
	uniformly for $N\geq M$. Therefore, 
	$$
	|\E f(\htc) - \E f(\htc\po)| = 
	\lim_{N\to\infty}|\E f(\htc^{(N)}) - \E f(\htc\po)|\leq p_M.
	$$ 
	By Lemma~\ref{lem:aux_exp} we have
	\[
	|\E f(\tc_n\po) - \E f(\htc\po)|\ll n^{-1/2} M^2 \E (\deg(\htc)^{\alpha+1})+n^{-1}M^2\E (\deg(\tc_n)^{\alpha+1}).
	\] 
	In view of \eqref{eq:prem3}, the moment $\E (\deg(\htc)^{\alpha+1})$ is finite and   $\E (\deg(\tc_n)^{\alpha+1})$ is $O(1)$ as $n\to\infty$. Therefore, we conclude that 
	$$
	|\E f(\tc_n) - \E f(\htc)|\ll p_M+n^{-1/2}M^2=p_{M_n}+n^{-1/2}{M_n}^2,
	$$
	which is equivalent to the statement in the lemma.
\end{proof}

Lemma~\ref{lem:expectation} is already enough to prove the estimate for the mean  $\E F(\tc_n)$ as it is stated in \eqref{eq:meanF}. This is a consequence of \cite[Theorem 1.5]{J16}.

\begin{proposition}\label{prop:mean}
Assuming that $f$ satisfies the conditions of Theorem~\ref{thm:normality}, then 
$$\E F(\tc_n)=n \mu +o(\sqrt{n}),$$
where $\mu=\E f(\tc)$.
\end{proposition}

\begin{proof}
Consider the shifted functional $F'(T)$ whose toll function is defined by $f'(T)=f(T)-\E f(\htc).$ Since $\E f(\htc)$ does not depend on $T$, we have 
\begin{equation}\label{eq:mean_proof}
\E F'(\tc_n)=\E F(\tc_n)-n \E f(\htc).
\end{equation}
Furthermore, notice that in the conditions of Theorem \ref{thm:normality}, \eqref{eq:thm1-2} and \eqref{eq:thm1-3} are unchanged if the toll function $f$ is shifted by a constant.  Hence, $f'$ also satisfies the conditions of Theorem \ref{thm:normality} where \eqref{eq:thm1-2} and \eqref{eq:thm1-3} hold with the same sequences $(p_M)_{M\geq 0}$ and $(M_n)_{n\geq 1}$. Thus, applying Lemma \ref{lem:expectation} to the toll function $f'$, we obtain
\[
\E f'(\tc_n) \ll p_{M_n} + n^{-1/2} \, M_n^2 ) \ll a_n,
\]
where $a_n$ is defined in Theorem \ref{thm:normality}. Since $a_n\to 0$ as $n\to\infty$, we have $\E f'(\tc_n)\to 0$ as $n\to\infty$. Therefore, by Part (i) of \cite[Theorem 1.5]{J16}, we deduce that 
\[
\E F'(\tc_n)=n\E f'(\tc)+o(\sqrt{n})=n\left(\E f(\tc)-\E f(\htc)\right)+o(\sqrt{n}).
\]
The latter and \eqref{eq:mean_proof} imply the result.
\end{proof}




Using the same notation as in Lemma~\ref{lem:Aux2} we obtain the following estimate of the variance.

\begin{lemma}\label{lem:variance}
	Assume that $f$ satisfies the conditions of Theorem~\ref{thm:normality} and let $(p_M)_{M\geq 1}$ and $(M_n)_{n}$ be the corresponding sequences. Moreover, set $a_k = k^{-1/2} (k^{\max\{\alpha,1\}}\, p_{M_k}+M_k^2)$ (as in Theorem~\ref{thm:normality}) and $\mu_k = \E f(\tc_k)$. Then, for any subset $\mathcal{S}$ of $\mathbb{N}$, we have
	\begin{equation}\label{eq:var}
	n^{-1/2}\Var\left( F_{\mathcal{S}}(\tc_n) \right)^{1/2} \ll \left(\sup_{k\in \mathcal{S}} a_k+\sum_{k\in  \mathcal{S}} \frac{a_k}{k}\right)^{1/2} + \sup_{k\in  \mathcal{S}} |\mu_k| + \sum_{k\in  \mathcal{S}}\frac{|\mu_k|}{k}.
	\end{equation}
\end{lemma}
\begin{proof}
	We follow the proof of \cite[Theorem 6.12]{J16}. We start with a decomposition $f_\mathcal{S}(T) = f_\mathcal{S}^{(0)}(T) + f_\mathcal{S}^{(1)}(T)$, where $f_\mathcal{S}^{(0)}(T) = f_\mathcal{S}(T) - \mu_{|T|}$ and $f_\mathcal{S}^{(1)}(T) = \mu_{|T|}$ if $|T|\in \mathcal{S}$ and both are zero otherwise. This induces a decomposition  $F_\mathcal{S}=F_\mathcal{S}^{(0)}+F_\mathcal{S}^{(1)}$ of the functional $F_\mathcal{S}$, where $F_\mathcal{S}^{(0)}$ and $F_\mathcal{S}^{(1)}$ are the additive functionals defined by the toll functions $f_\mathcal{S}^{(0)}$ and $f_\mathcal{S}^{(1)}$ respectively. In view of Minkowski's inequality $\Var (X + Y)^{1/2} \le \Var (X)^{1/2} + \Var (Y)^{1/2}$, we can estimate the variances $\Var F_\mathcal{S}^{(0)}(\tc_n)$ and $\Var F_\mathcal{S}^{(1)}(\tc_n)$ separately. Note the following important observations:
	\begin{enumerate}
		\item [(a)]  $f_\mathcal{S}^{(1)}(T)$ depends on $|T|$ only,
		\item[(b)]  $\E f_\mathcal{S}^{(0)}(\tc_k) = 0$ for every $k\geq 1.$
	\end{enumerate}
	If a toll function of an additive functional depends only on tree sizes, then a bound on the variance of the corresponding additive functional at $\tc_n$ is given in \cite[Theorem 6.7]{J16}. This applies to our toll function $f_\mathcal{S}^{(1)}$, and we obtain
	\begin{equation}\label{eq:var_first}
	n^{-1/2}\Var\left( F_\mathcal{S}^{(1)}(\tc_n) \right)^{1/2} \ll  \sup_{k\geq 1 } |\E f_\mathcal{S}^{(1)}(\tc_k)| + \sum_{k=1}^{\infty} \frac{|\E f_\mathcal{S}^{(1)}(\tc_k)| }{k} = \sup_{k\in \mathcal{S}} |\mu_k| + \sum_{k\in \mathcal{S}} \frac{|\mu_k|}{k} . 
	\end{equation}
Next, we  consider $\Var (F_\mathcal{S}^{(0)}(\tc_n))$. By \cite[(6.28)]{J16}, we have
	\begin{equation}\label{eq:variance0}
	\frac 1 n \Var\left( F_\mathcal{S}^{(0)}(\tc_n) \right) \le 2\sum_{k=1}^n\frac{\prob{S_{n-k} = n - k}}{\prob{S_n = n - 1}} \pi_k \E (f_\mathcal{S}^{(0)}(\tc_k) F_\mathcal{S}^{(0)}(\tc_k)),
	\end{equation}
	where $\pi_k = \prob{|\tc| = k}$, and $S_k$ is the sum of $k$  independent copies of $\xi$. From \cite[Lemma 5.2]{J16}, we know that  
	$$
	\frac{\prob{S_{n-k} = n - k}}{\prob{S_n = n - 1}} \ll \frac{n^{1/2}}{(n-k+1)^{1/2}},
	$$
	uniformly for $1\leq k\leq n$. Recalling that $\pi_k = O (k^{-3/2})$, which is a well-known fact but can also be found in \cite[(4.13)]{J16},  we obtain 
	\begin{equation}\label{eq:variance}
	\frac 1 n \Var\left( F_\mathcal{S}^{(0)}(\tc_n) \right) \ll \sum_{k=1}^n\frac{n^{1/2}}{(n-k+1)^{1/2}\, k^{3/2}}\E (f_\mathcal{S}^{(0)}(\tc_k) F_\mathcal{S}^{(0)}(\tc_k)).
	\end{equation}
	For  $k\notin \mathcal{S}$, we have $\E (f_\mathcal{S}^{(0)}(\tc_k) F_\mathcal{S}^{(0)}(\tc_k))=0$, and  for $k\in \mathcal{S}$ we use  Lemma \ref{lem:Aux2} to estimate $\E (f_\mathcal{S}^{(0)}(\tc_k) F_\mathcal{S}^{(0)}(\tc_k))$. Once again, by means of the second estimate in \eqref{eq:prem3}, both  $\E(\deg(\tc_{k})^{2\alpha})$ and $\E(\deg(\tc_{k})^{\alpha+1})$ are bounded above by constants. Thus, for $k\in \mathcal{S}$, we deduce that
	\begin{equation}\label{eq:fF}
	\E (f_\mathcal{S}^{(0)}(\tc_k) F_\mathcal{S}^{(0)}(\tc_k))\ll k^{\max\{\alpha,\,1\}} \, p_{M_k}+M_k^2 =k^{1/2}a_k.
	\end{equation}
 Applying \eqref{eq:fF} to \eqref{eq:variance}, we get
	\begin{align*}
	\frac 1 n \Var\left( F_\mathcal{S}^{(0)}(\tc_n) \right) 
	& \ll 
	\sum_{k=1}^n \frac{n^{1/2}\, a_k\, \I_{\{k\in \mathcal{S}\}}}{(n-k+1)^{1/2}\,k}  \\
	& \ll \sum_{k = 1}^{n/2} \frac{a_k}{k}\I_{\{k\in \mathcal{S}\}} + \sup_{k \in \mathcal{S}} a_k \sum_{n/2\leq k\leq n} \frac{1}{(n-k+1)^{1/2}\,n^{1/2}}.
	\end{align*}
Noting that the last sum on the right side is bounded by a constant, we obtain
\begin{equation}\label{eq:var_last}
\frac 1 n \Var\left( F_\mathcal{S}^{(0)}(\tc_n) \right) \ll \sup_{k \in \mathcal{S}} a_k + \sum_{k\in \mathcal{S}}\frac{a_k}{k}.
\end{equation}
The proof is complete by applying  Minkowski's inequality to combine \eqref{eq:var_first} and \eqref{eq:var_last}.
\end{proof}

\section{Central limit theorem}
We use a truncation argument as in the proof of \cite[Theorem 1.5]{J16}. This is formulated in the following lemma: 
\begin{lemma}\label{lem:0}
	Let $(X_n)_{n\geq 1}$ and $(W_{N,n})_{N,n\geq 1}$ be sequences of centred random variables. If we have  
	\begin{itemize}
		\item $W_{N,n}\overset{d}{\to}_n W_N$, and  $ W_N \overset{d}{\to}_N W,$ for some random variables $W,$ $W_1$, $W_2$, \dots 
		\item $\mathrm{Var}(X_n-W_{N,n})=O(\sigma^2_N)$ uniformly in $n$, and $\sigma^2_N\to_N0$,
	\end{itemize}
	then 
	$
	X_n\overset{d}{\to}_n W.
	$
\end{lemma}
This lemma is a simple consequence of \cite[Theorem 4.28]{K02} or \cite[Theorem 4.2]{B68}.

\begin{proof}[Proof of Theorem~\ref{thm:normality}]
	We may assume, without loss of generality, that $\E f(\htc)=0$, by subtracting $\E f(\htc)$ from $f$ if it is not zero, because shifting $f$ by a constant will only add a deterministic term in $F(\tc_n)$ and $f$ still satisfies the conditions of Theorem~\ref{thm:normality} where the sequences $(p_M)_{M\geq 1}$ and $(M_n)_{n\geq 1}$ remain the same.  For each $k$, let $\mu_k$ denote the expectation $\E f(\tc_k)$ as before. By Lemma~\ref{lem:expectation}, we have
	\begin{equation}\label{eq:clt_a}
	|\mu_k|=|\E f(\tc_k)|\ll p_{M_k}+k^{-1/2}M_k^2\leq a_k. 
	\end{equation}
	For a positive integer $N$, let $f^{(N)}$ be the truncated functional defined by $f^{(N)}(T)=f(T)\, \I_{\{|T|<N\}}$ (i.e. $f^{(N)}(T)=f(T)$ if $|T|<N$, and $f^{(N)}(T)=0$ otherwise) and let $F^{(N)}$ be the additive functional associated with the toll function $f^{(N)}$. It is important to notice that $f^{(N)}$ is local, for any fixed $N$.  Note further that $\E f^{(N)}(\tc_k)= \mu_k$ if $k<N$, and zero otherwise. Hence, we have $|\E f^{(N)}(\tc_k)|\leq |\mu_k|$ for all positive integers $N$ and $k$. Let
	$$
	W_{N,n}:= \frac{F^{(N)}(\tc_n)-\E F^{(N)}(\tc_n)}{\sqrt{n}}, \, \text{ and }\, X_n:=\frac{F(\tc_n)-\E F(\tc_n)}{\sqrt{n}}.
	$$
	Since $f^{(N)}$ has finite support, by \cite[Theorem 1.5]{J16},  we have 
	$$
	W_{N,n}\overset{d}{\to}_n  \mathcal{N}(0,\gamma_N^2),
	$$ 
	where
	\begin{align*}
		\gamma_N^2
		& =\lim_{n\to\infty}n^{-1}\mathrm{Var}(F^{(N)}(\tc_n))\\
		&=2 \E \left(f^{(N)}(\tc)\, (F^{(N)}(\tc)-|\tc|\mu^{(N)})\right)-\mathrm{Var}f^{(N)}(\tc)-\frac{(\mu^{(N)})^2}{\sigma^2}, 
	\end{align*}
	and $\mu^{(N)}=\E f^{(N)}(\tc)$. 
	
	\medskip
	Next we need to show that $\lim_{N\to\infty}\gamma_N$ exists. To that end, we take an arbitrary integer $M\geq N$. We have
	$$
	\gamma_M-\gamma_N = \lim_{n\to\infty}n^{-1/2}\left(\mathrm{Var}(F^{(M)}(\tc_n))^{1/2}-\mathrm{Var}(F^{(N)}(\tc_n))^{1/2}\right)
	$$ 
	If we apply Minkowski's inequality to the random variables $F^{(M)}(\tc_n)-F^{(N)}(\tc_n)$ and $F^{(N)}(\tc_n)$, we obtain
	$$
	\mathrm{Var}(F^{(M)}(\tc_n))^{1/2}\leq \mathrm{Var}\Big(F^{(M)}(\tc_n)-F^{(N)}(\tc_n)\Big)^{1/2}+\mathrm{Var}(F^{(N)}(\tc_n))^{1/2}.
	$$
	Consequently,
	\begin{align*}
		|\gamma_M-\gamma_N|
		& = \lim_{n\to\infty}n^{-1/2} |\mathrm{Var}(F^{(M)}(\tc_n))^{1/2}-\mathrm{Var}(F^{(N)}(\tc_n))^{1/2}|\\
		& \leq \limsup_{n\to\infty}n^{-1/2}\mathrm{Var}\Big(F^{(M)}(\tc_n)-F^{(N)}(\tc_n)\Big)^{1/2}.
	\end{align*}
	The toll function associated with the functional $F^{(M)}-F^{(N)}$ is $f^{(M)}-f^{(N)}$, which is zero for all trees of order smaller than $N$. Hence, Lemma~\ref{lem:variance} can be used to estimate the variance  $\mathrm{Var}(F^{(M)}(\tc_n)-F^{(N)}(\tc_n))^{1/2}$ (for this the set $\mathcal{S}$ in Lemma~\ref{lem:variance} is chosen to be $\mathbb{N} \cap [N,\, M)$).  We obtain
	\begin{align*}
		|\gamma_M-\gamma_N|
		&  \ll \left(\sup_{k\geq N} a_k+\sum_{k= N}^{\infty} \frac{a_k}{k}\right)^{1/2} + \sup_{k\geq N} |\mu_k| + \sum_{k=N}^\infty \frac{|\mu_k|}{k}\\
		& \ll \left(\sup_{k\geq N} a_k+\sum_{k= N}^{\infty} \frac{a_k}{k}\right)^{1/2} + \sup_{k\geq N} a_k + \sum_{k=N}^\infty \frac{a_k}{k}.
	\end{align*}
	The last line follows from \eqref{eq:clt_a}. By condition \eqref{eq:thm1-4} of Theorem~\ref{thm:normality}, we also deduce that $|\gamma_M-\gamma_N|\to_N0$ uniformly for $M\geq N$. Hence, the sequence $(\gamma_N)_N$ is a Cauchy sequence, which implies that  $\gamma:=\lim_{N\to\infty}\gamma_N$ exists.
	
Similarly, we have 
	\begin{align*}
		\mathrm{Var}(X_n-W_{N,n})^{1/2}
		& =n^{-1/2}\mathrm{Var}(F(\tc_n)-F^{(N)}(\tc_n))^{1/2}.
	\end{align*}
Once again, Lemma~\ref{lem:variance} applies here, where the set $\mathcal{S}$ is $\mathbb{N} \cap [N,\, \infty)$. We obtain 
\[
n^{-1/2}\mathrm{Var}(F(\tc_n)-F^{(N)}(\tc_n))^{1/2}\ll \left(\sup_{k\geq N} a_k+\sum_{k= N}^{\infty} \frac{a_k}{k}\right)^{1/2} + \sup_{k\geq N} a_k + \sum_{k=N}^\infty \frac{a_k}{k}.
\]
Therefore, we conclude that $\mathrm{Var}(X_n-W_{N,n})^{1/2}$ tends to zero as $N\to\infty$ uniformly in $n$, so Lemma~\ref{lem:0} applies. With the estimate of $\E F(\tc_n)$ in \eqref{eq:meanF} which is proved in Proposition~\ref{prop:mean}, we obtain \eqref{eq:normality}
  and the proof of Theorem~\ref{thm:normality} is complete. \end{proof}
\section{Examples}

In this section, we give several applications of our main theorem. Every local functional (as defined in the introduction) trivially satisfies the conditions of Theorem~\ref{thm:normality}. This already gives us a number of examples to which the theorem applies, for instance the number of nodes of outdegree $r$ for every fixed $r$, or more generally the number of nodes whose outdegree lies in some prescribed set $R$. However, we want to focus on functionals in this section that are not covered by any previous results. The first example is treated in detail in Subsection~\ref{sub:ind}, where we prove a non-degenerate central limit theorem for the logarithm of the number independent sets in $\tc_n.$ In the other examples, we only verify that the conditions of Theorem~\ref{thm:normality} are satisfied by the corresponding functionals, without proving the non-degeneracy in each case (however, the approach of our first example can be applied to the others as well).   

\subsection{The number of independent sets}\label{sub:ind} An independent set is a set of vertices which does not contain two vertices that are adjacent. The number of independent sets (this number is also known as \emph{Fibonacci} number of $T$; see \cite{KPT86}) was studied in the random plane graph by Kirschenhofer, Prodinger and Tichy \cite{KPT86} who determined the formula for its expectation (see also \cite{K97,W07b}). However, in order to obtain a limiting distribution, one has to study the logarithm of the number of independent sets rather than the number itself, as we will see in the following.

Let $I(T)$ be the total number of independent sets of $T$ and $I_0(T)$ be the number of independent sets of $T$ that do not contain the root.
The quantities $I$ and $I_0$ satisfy the following recursive formulas, where $T_1, \dots, T_{\deg(T)}$ stand for the root branches:
\begin{align}
I_0(T) & = \prod_{i}I(T_i), \label{eq:ind1}\\
I(T) & = I_0(T)+\prod_{i}I_0(T_i).\label{eq:ind2}
\end{align}
The first identity holds since every independent set of $T$ that does not contain the root uniquely decomposes into independent sets in the branches. The second identity holds for essentially the same reason, taking into account those independent sets of $T$ that contain the root, which can therefore not contain any of the roots of the branches.

Note that \eqref{eq:ind1} and \eqref{eq:ind2} are also satisfied by $T = \bullet$, a tree consisting of a single node.
Anticipating a log-normal limit distribution, we define an additive functional $F(T) := \log I(T)$. From \eqref{eq:ind1} it follows that the associated toll function is
\begin{equation}\label{eq:ind_f}
f(T)=F(T) - \sum_i F(T_i) = \log \left(\frac{I(T)}{\prod_i I(T_i)}\right) = \log \left( \frac{I(T)}{I_0(T)} \right) =  \log \left(1 + \prod_i \frac{I_0(T_i)}{I(T_i)}\right).
\end{equation}
Since $I_0(T_i) \le I(T_i)$, it follows immediately that $0\leq f(T) \leq \log 2$. Hence, the condition  \eqref{eq:thm1-1} of Theorem~\ref{thm:normality} is satisfied with $\alpha=0$. 

Further, let $\rho(T) := \frac{I_0(T)}{I(T)}$.
By \eqref{eq:ind1} and \eqref{eq:ind2}, functional $\rho$ also satisfies a recursion, namely
\begin{equation}\label{eq:ind3}
\rho(T)=\frac{1}{1+\prod_{i}\rho(T_i)}.
\end{equation}
Observe that \eqref{eq:ind1}, \eqref{eq:ind2}, and \eqref{eq:ind_f} imply
\begin{equation}
	\label{eq:ind_f_rho}
	f(T) = - \log \rho(T).
\end{equation}
In order to measure the difference between $f(T)$ and $f(T\po)$ in terms of $M$, we define the exact bounds on $\rho$ given the first~$M$ levels:
\[
\rho^M_{\inf}(T) := \inf \{\rho(S) : S^{(M)} = T^{(M)}\}, \qquad \rho^M_{\sup}(T) := \sup \{\rho(S) : S^{(M)} = T^{(M)}\}.
\]
From \eqref{eq:ind_f_rho} it follows that for any tree $T$
\begin{equation}\label{eq:ind_ftau}
|f(T)-f(T\po)|\leq \log(\rho^M_{\sup}(T)/\rho^M_{\inf}(T)) =: \tau^{M}(T).
\end{equation}
In view of \eqref{eq:ind3} we have the trivial bounds $1/2 \le \rho^0_{\inf}(T) \le \rho^0_{\sup}(T) \le 1$, which imply
\begin{equation}\label{eq:ind_tautriv}
\tau^0(T) \le \log 2.
\end{equation}
For $M \ge 1$ the functions $\rho^M_{\inf}(T)$ and $\rho^M_{\sup}(T)$ can be determined recursively using \eqref{eq:ind3}, which gives
\begin{equation}\label{eq:ind_recursion}
 \rho^M_{\sup}(T)=\frac{1}{1+\prod_{i}\rho^{M-1}_{\inf}(T_i)} \, \text{ and }\, \rho^M_{\inf}(T)=\frac{1}{1+\prod_{i}\rho^{M-1}_{\sup}(T_i)}.
\end{equation}
 Using \eqref{eq:ind_recursion} and writing $\Pi := \prod_{i}\rho^{M-1}_{\sup}(T_i)$ and $\Sigma := \sum_j\tau^{M-1}(T_j)$, we get
 \begin{equation*}
\tau^M(T)  = \log\left(\frac{1+\prod_{i}\rho^{M-1}_{\sup}(T_i)}{1+\prod_{i}\rho^{M-1}_{\inf}(T_i)}\right) 
= -\log\left(\frac{1 + e^{-\Sigma}\Pi}{1 + \Pi}\right).
 \end{equation*}
 Applying Jensen's inequality to the convex function $x \mapsto - \log x$ and using $\Pi \le 1$, we infer
 \begin{equation}\label{eq:ind_taurecursion}
	 \tau^M(T) \le \frac{\Pi}{1 + \Pi} \Sigma \le \frac{1}{2}\sum_i \tau^{M-1}(T_i).
 \end{equation}
 Let  $v_1,\, v_2,\,\dots,\, v_{w_{M}(T)}$ be the nodes of $T$ at level $M$. By applying \eqref{eq:ind_taurecursion} recursively $M$ times and using \eqref{eq:ind_tautriv}, we obtain a bound
\begin{equation}\label{eq:ind5}
\tau^{M}(T)\leq 2^{-M} \sum_{i=1}^{w_{M}(T)}\tau^0(T_{v_i})
= \frac{\log 2}{2^M}w_{M}(T) .
\end{equation}
Combining \eqref{eq:ind_ftau} and \eqref{eq:ind5} we obtain
\begin{equation}
	\label{eq:ind_fwM}
	|f(T) - f(T^{(M)})| \le \frac{\log 2}{2^M}w_{M}(T).
\end{equation}
Now we are ready to verify that the remaining conditions of Theorem~\ref{thm:normality} are satisfied by our toll function. Note that for any $N\geq M$, we have 
\[
\E \left|  f(\hat \tc^{(M)})-\E\left(f(\hat \tc^{(N)})\,  | \, \hat \tc^{(M)}\right)\right|\leq \E \left(\E \left(|f(\htc\po)-f(\htc^{(N)})|\, \Big| \, \htc\po\right)\right).
\]
Using \eqref{eq:ind_fwM}, we deduce that for any $N\geq M$,
\[
\E \left(|f(\htc\po)-f(\htc^{(N)})|\, \Big| \, \htc\po\right)\leq \frac{\log 2}{2^M}\E \left(w_M(\htc^{(N)})\, \Big| \, \htc\po\right).
\]
By taking the expectations, and using $w_M(\htc^{(N)}) = w_M(\htc)$ as well as the estimate $\E w_M(\htc)=O(M)$ (see \eqref{eq:prem1}), we get 
\begin{equation}\label{eq:ind7}
\E \left|  f(\hat \tc^{(M)})-\E\left(f(\hat \tc^{(N)})\,  | \, \hat \tc^{(M)}\right)\right|\ll  M\, 2^{-M}.
\end{equation}
To check the condition \eqref{eq:thm1-3} we use \eqref{eq:ind_fwM} and $\E w_M(\tc_n)=O(M)$ (see \eqref{eq:prem1}) and get
\begin{multline}\label{eq:ind8}
\E |f(\tc_n)-f(\tc_n\po)| =  \E \left(\E \left(|f(\tc_n)-f(\tc_n\po)| \, \Big|\, \tc_n\po \right)\right) \\
\le \E \left( \frac{\log 2}{2^M} w_M(\tc_n) \right) \ll M 2^{-M},
\end{multline}
where the implied constant is independent of $n$. To sum up, \eqref{eq:ind7} and \eqref{eq:ind8} show that assumptions \eqref{eq:thm1-2} and \eqref{eq:thm1-3} of Theorem~\ref{thm:normality} hold for a suitable choice of $p_M$ and $M_n$ with $p_M \ll M2^{-M}$  and $M_n \ll \log n$, which implies that condition~\eqref{eq:thm1-4} is satisfied.

In the following, we show that the variance constant $\gamma$ in Theorem~\ref{thm:normality} is always strictly positive for the functional $F(T) = \log I(T)$. The approach that we use also applies (mutatis mutandis) to our other examples in the following sections, so we will not explicitly prove positivity of $\gamma$ in all those cases.

As a first step, choose two trees $S_1$ and $S_2$ with the same number of vertices that both have a positive probability, i.e. $\Prob (\tc = S_1) > 0$ and $\Prob (\tc = S_2) > 0$, and also satisfy $I(S_1) > I(S_2)$ and $I_0(S_1) > I_0(S_2)$. This is always possible, for example in the following way: let $d$ be a possible outdegree for the given offspring distribution, i.e. $\Prob(\xi = d) > 0$. Now let $S_1$ be a complete $d$-ary tree of height $3$ (the root has $d$ children, each of which has $d$ children, each of which has again $d$ children, which are leaves), and let $S_2$ be a $d$-ary caterpillar with the same number of vertices, consisting of $d^2+d+1$ internal vertices that form a path, and $d^3$ leaves (each internal node has $d-1$ leaf children, except for the last, which has $d$ leaf children). One can verify that both inequalities hold for this choice of $S_1$ and $S_2$ for all $d$.

The key observation is that replacing a fringe subtree isomorpic to $S_2$ in a tree by $S_1$ increases the number of independent sets by at least a fixed factor greater than $1$. To see this, suppose that $S$ is a fringe subtree of a tree $T$ rooted at $r$, let $T'$ be the tree obtained by removing the entire fringe subtree $S$ from $T$, and let $v$ be the parent of $r$ in $T'$. If $A$ is the number of independent sets of $T'$ that do not contain $v$, and $B$ the number of independent sets of $T'$ that contain $v$, then we have (distinguishing independent sets containing and not containing $v$)
$$I(T) = A I(S) + B I_0(S).$$
As a consequence of this representation, we find that the values of $I(T)$ for $S = S_1$ and $S = S_2$ differ at least by a factor of $\eta = \min \{ \frac{I(S_1)}{I(S_2)}, \frac{I_0(S_1)}{I_0(S_2)} \} > 1$ (and at most by $\max \{ \frac{I(S_1)}{I(S_2)}, \frac{I_0(S_1)}{I_0(S_2)} \}$). 

Now consider a large random tree $\tc_n$ with $n$ vertices. We replace each occurrence of $S_1$ or $S_2$ as a fringe subtree by a marked leaf. The resulting tree, which has some number of marked leaves, is denoted by $\tc_n^*$. Given that $\tc_n^*$ has marked leaves $v_1,v_2,\ldots,v_m$, the original tree $\tc_n$ is obtained by replacing each marked node $v_i$ by a tree $R_i \in \{S_1,S_2\}$. Conditioned on the shape of $\tc_n^*$, the different fringe subtrees $R_i$ are all independent, and the probabilities $\Prob(R_i = S_1) = p > 0$ and $\Prob(R_i = S_2) = q = 1-p > 0$ are independent of $i$ and $\tc_n^*$; they only depend on the choice of $S_1$ and $S_2$.

Still conditioning on the shape of $\tc_n^*$, we would like to determine a lower bound for the variance of $F(\tc_n) = \log I(\tc_n)$. Iterated application of the law of total variance yields

\begin{align*}
\Var (F(\tc_n) | \tc_n^*) &= \E \Big( \Var( F(\tc_n) | \tc_n^*, R_1,R_2,\ldots,R_{m-1} ) | \tc_n^* \Big) \\
&\quad+ \sum_{j=2}^{m-1} 
\E \Big( \Var \big( \E ( F(\tc_n) | \tc_n^*, R_1,R_2,\ldots,R_j ) | \tc_n^*, R_1,R_2,\ldots,R_{j-1} \big) | \tc_n^* \Big) \\
&\quad+ \Var \big( \E ( F(\tc_n) | \tc_n^*, R_1 ) | \tc_n^* \big).
\end{align*}
As mentioned before, replacing a fringe subtree isomorphic to $S_2$ by a fringe subtree isomorphic to $S_1$ increases the number of independent sets at least by a factor $\eta$ (thus increases the logarithm by at least $\log \eta$), regardless of the remaining shape of the tree. Therefore, each $R_i$ contributes at least the fixed constant $p(1-p)\log^2 \eta$ to the variance decomposition above, which shows that
$$\Var (F(\tc_n) | \tc_n^*)  \gg m,$$
uniformly for all possible shapes of $\tc_n^*$. Applying the law of total variance once again and recalling that $m$ is the total number of fringe subtrees isomorphic to $S_1$ or $S_2$, we find that
$$\Var (F(\tc_n) ) \gg \E ( F_{S_1} (\tc_n)) + \E ( F_{S_2} (\tc_n)).$$
The two functionals $F_{S_1}$ and $F_{S_2}$, counting fringe subtrees isomorphic to $S_1$ and $S_2$ respectively, are additive functionals whose means are linear in $n$ with nonzero constants by our choice of $S_1$ and $S_2$ (see for example \cite[(1.10)]{J16}). Therefore, it follows that
$$\Var (F(\tc_n) ) \gg n,$$
which shows that $\gamma > 0$. Thus we have a non-degenerate central limit theorem for the logarithm of the number of independent vertices. Let us formulate this as a theorem:

\begin{theorem}
Let $\tc_n$ be a conditioned Galton-Watson tree of order $n$ with offspring distribution $\xi$, where $\xi$ satisfies $\E \xi = 1$ and $0<\sigma^2:=\mathrm{Var} \xi <\infty$. There exist constants $\mu > 0$ and $\gamma > 0$ (both depending on $\xi$) such that
$$\frac{\log I (\tc_n) - n\mu}{\sqrt n} \dto \nc(0,\gamma^2)$$
as $n \to \infty$.
\end{theorem}

\subsection{The number of matchings}

The number of matchings in random trees has been studied previously, and means and variances have been determined for different classes of trees \cite{K97,W07a,W07b}. Just like in the previous example, in order to obtain a limiting distribution, we will consider the logarithm of this quantity. The proof is very similar to the previous example. For a rooted tree $T$, let $m(T)$ be the total number of matchings of  $T$ and $m_0(T)$ be the number of matchings of $T$ that do not cover the root (by this, we mean matchings that do not contain an edge incident to the root). Using similar arguments as for the number of independent sets, one finds that these functionals satisfy the following recursive formulas (including the case when $T$ consists of a single node):
\begin{align}
m_0(T) & = \prod_{i}m(T_i), \label{eq:match1}\\
m(T) & = m_0(T)+\sum_{i} m_0(T_i) \prod_{j\neq i}m(T_j).\label{eq:match2}
\end{align}
Defining an additive functional $F(T) := \log m(T)$, we observe from \eqref{eq:match1} and \eqref{eq:match2} that the associated toll function is
\begin{equation}\label{eq:matchf}
f(T)=F(T) - \sum_i F(T_i) = \log m(T) - \sum_i \log m(T_i) = - \log\left(\frac{m_0(T)}{m(T)}\right).
\end{equation}
We define $\rho(T) :=\frac{m_0(T)}{m(T)}$, which, by \eqref{eq:match1} and \eqref{eq:match2}, also satisfies a recursion, namely
\begin{equation}\label{eq:match3}
\rho(T)=\frac{1}{1+\sum_{i}\rho(T_i)}.
\end{equation}
From \eqref{eq:matchf} it follows that $f(T) = - \log \rho(T)$, which, in view of \eqref{eq:match3}, implies that $0\leq f(T)\leq \log(1+\deg(T))$. Hence, condition  \eqref{eq:thm1-1} of Theorem~\ref{thm:normality} is satisfied by $f$ with $\alpha=1$, say. 

To estimate the distance between $T$ and $T^{(M)}$, we define
\[
\rho^M_{\inf}(T) := \inf \{\rho(S) : S^{(M)} = T^{(M)}\}, \qquad \rho^M_{\sup}(T) := \sup \{\rho(S) : S^{(M)} = T^{(M)}\},
\]
so that
\begin{equation}\label{eq:ftau}
|f(T)-f(T\po)|\le \frac{\rho^M_{\sup}(T)}{\rho^M_{\inf}(T)} =: \tau^{M}(T).
\end{equation}
The functionals $\rho^M_{\inf}(T)$ and $\rho^M_{\sup}(T)$, $M = 0,1, 2, \dots$ satisfy recursions
\begin{equation}\label{eq:match_recursion}
 \rho^M_{\sup}(T)=\frac{1}{1+\sum_{i}\rho^{M-1}_{\inf}(T_i)} \, \text{ and }\, \rho^M_{\inf}(T)=\frac{1}{1+\sum_{i}\rho^{M-1}_{\sup}(T_i)}.
\end{equation}
 Using \eqref{eq:match_recursion}, and denoting $\rho_i = \rho_{\inf}^{M-1}(T_i)$, we get
\[
\tau^M(T)  = -\log\left(\frac{1+\sum_{i}\rho^{M-1}_{\inf}(T_i)}{1+\sum_{i}\rho^{M-1}_{\sup}(T_i)}\right) 
 = -\log\left(\frac{1+\sum_{i}\rho_i\exp(-\tau^{M-1}(T_i))}{1+\sum_{i}\rho_i}\right).
\]
Since the argument of the logarithm on the right side is a convex combination of expressions $\exp(-\tau^{M-1}(T_i))$, $i = 1, 2, \dots$, applying Jensen's inequality to the convex function $x \mapsto -\log x$ yields, for $T$ with $\deg(T) \ge 1$,
\begin{equation}\label{eq:taurecursion}
\tau^{M}(T) 
 \leq \frac{1}{1+\sum_i \rho_i}\sum_i\rho_i\tau^{M-1}(T_i)\\[1em]
 \leq \frac{\max_i \rho_i}{1+\max_i \rho_i}\sum_{i}\tau^{M-1}(T_i)
\leq \frac{1}{2}\sum_{i}\tau^{M-1}(T_i).
\end{equation}
But for $M \ge 1$, inequality $\tau^M(T) \le \frac{1}{2}\sum_i \tau^{M-1}(T_i)$ is also satisfied for the only tree $T$ with $\deg(T) = 0$.
Unlike the previous example involving the number of independent sets, $\tau^0$ is not bounded by a constant, but the situation is saved by bounding $\tau^1$ by the root degree instead. From \eqref{eq:match_recursion} it is clear that $\rho^1_{\sup}(T)=1$ and $\rho^1_{\inf}(T)=(1+\deg(T))^{-1}$ for every $T$. Therefore 
\begin{equation}\label{eq:taudeg}
\tau^1(T) = \log (1 + \deg(T)) \le \deg(T). 
\end{equation}
Let  $v_1,\, v_2,\,\dots,\, v_{w_{M-1}(T)}$ be the nodes at level $M - 1$ of $T$. By iterating \eqref{eq:taurecursion} $M - 1$ times and applying \eqref{eq:taudeg}, we obtain
\begin{equation}\label{eq:match5}
\tau^{M}(T)\leq 2^{-(M - 1)} \sum_{i=1}^{w_{M-1}(T)}\tau^1(T_{v_i})
\leq 2^{-(M - 1)} \sum_{i=1}^{w_{M-1}(T)} \deg(T_{v_i}) \le 2^{-(M-1)} w_M(T).
\end{equation}
Combining \eqref {eq:ftau} and \eqref{eq:match5}, we obtain
\begin{equation}\label{eq:match6}
|f(T)-f(T\po)|\leq 2^{-M + 1 }w_{M}(T).
\end{equation}
Since \eqref{eq:match6} differs from the corresponding inequality \eqref{eq:ind_fwM} for the number of independent sets just by a constant factor, checking of the conditions \eqref{eq:thm1-2} and \eqref{eq:thm1-3} works precisely in the same way and shows that the conditions of Theorem~\ref{thm:normality} are again satisfied for some choice of $p_M, M_n$ satisfying $p_M \ll M2^{-M}$  and $M_n \ll \log n$.   

We conclude again with a formal theorem:

\begin{theorem}
Let $\tc_n$ be a conditioned Galton-Watson tree of order $n$ with offspring distribution $\xi$, where $\xi$ satisfies $\E \xi = 1$ and $0<\sigma^2:=\mathrm{Var} \xi <\infty$ as well as $\E \xi^3 < \infty$. There exist constants $\mu > 0$ and $\gamma > 0$ (both depending on $\xi$) such that
$$\frac{\log m (\tc_n) - n\mu}{\sqrt n} \dto \nc(0,\gamma^2)$$
as $n \to \infty$.
\end{theorem}

\subsection{The number of dominating sets} Recall that a dominating set $D \subseteq V(T)$ is a set of nodes so that every node of the tree is either in $D$ or is has a neighbour in $D$. Let $d(T)$ be the number of dominating sets in $T$. Extreme values of $d$ in trees were studied by Br\'od and Skupie\'n \cite{BS06}.

Applying Theorem~\ref{thm:normality} to the number of dominating sets $d(T)$ is more complicated than the cases of independent sets and matchings. We consider two auxiliary parameters, defining $d_0(T)$ to be the number of dominating sets not containing the root and $d_*(T)$ to be the number of sets dominating everything except for the root (in particular such a set contains neither the root nor a child of the root). A bit of consideration reveals that the following recursive formulas are satisfied for any tree $T$ with root branches $T_1, \dots, T_{\deg(T)}$, the reasoning being similar to the previous two examples: 
\begin{align*}
d_*(T) & =\prod_id_0(T_i), \\
d_0(T) & =\prod_id(T_i)-\prod_id_0(T_i), \\
d(T) & = d_0(T) + \prod_i\left(d(T_i)+d_*(T_i)\right).
\end{align*} 
Considering an additive functional $F(T)=\log d(T)$, we get that the corresponding toll function is 
\begin{equation}\label{eq:domf}
	f(T)= F(T) - \sum_i F(T_i) = \log \frac{d(T)}{\prod_i d(T_i)}	= \log \frac{d(T)}{d_0(T) + \prod_i d_0(T_i)} = -\log \frac{d_0(T)+d_*(T)}{d(T)}.
\end{equation}
Defining
\[
	\rho_*(T):= \frac{d_*(T)}{d(T)} \quad \text{and} \quad \rho_0(T) := \frac{d_0(T)}{d(T)},
\]
from \eqref{eq:domf} we obtain
\begin{equation}\label{eq:dom_frho}
		f(T) = -\log \Big(\rho_*(T)+\rho_0(T)\Big).
\end{equation}
It is easy to see that $0\leq \rho_*(T)\leq 1$ and $0\leq \rho_0(T)\leq \frac{1}{2}$. This implies $f(T) \ge - \log (3/2)$. On the other hand, using the recursive formulas for $d_0$ and $d$ as well as $\rho_* \le 1$, we get
\begin{multline}\label{eq:dom_fdeg}
f(T) = \log \frac{d(T)}{\prod_i d(T_i)} = \log \frac{\prod_i d(T_i) - \prod_i d_0(T_i) + \prod_i \left(d(T_i)+d_*(T_i)\right)}{\prod_i d(T_i)} \\
\le
 \log \left(1 + \prod_i \left(1 + \rho_*(T_i)\right)\right)
\le \log \left(1 + 2^{\deg(T)}\right) \le \log 2 \, \deg(T) + 1.
\end{multline}
Hence condition \eqref{eq:thm1-1} is satisfied with $\alpha = 1$.

Preparing to estimate the distance between $f(T)$ and $f(T^{(M)})$, we first note that the functionals $\rho_*$ and $\rho_0$ satisfy the recursions
\begin{equation}\label{eq:rhostarrec}
  \rho_*(T) =  \frac{\prod_i\rho_0(T_i)}{1 - \prod_i\rho_0(T_i) + \prod_i \left( 1 +\rho_*(T_i) \right)}, 
\end{equation}
\begin{equation}\label{eq:rho0rec}
  \rho_0(T) = \frac{1 - \prod_i \rho_0(T_i)}{1 - \prod_i \rho_0(T_i) + \prod_i \left( 1 +\rho_*(T_i) \right)}.
\end{equation}
We further define
\begin{equation}\label{eq:dom_rhoM0}
	\rho^M_{0,\inf}(T) := \inf \{\rho_0(S) : S^{(M)} = T^{(M)}\}, \qquad \rho^M_{0,\sup}(T) := \sup \{\rho_0(S) : S^{(M)} = T^{(M)}\},
\end{equation}
\begin{equation}\label{eq:dom_rhoMstar}
	\rho^M_{*,\inf}(T) := \inf \{\rho_*(S) : S^{(M)} = T^{(M)}\}, \qquad \rho^M_{*,\sup}(T) := \sup \{\rho_*(S) : S^{(M)} = T^{(M)}\}.
\end{equation}
Let
\[
\tau_0^M(T) := \log \frac{\rho_{0,\sup}^M(T)}{\rho_{0,\inf}^M(T)},
\qquad
\tau_*^M(T) = \log \frac{\rho_{*,\sup}^M(T)}{\rho_{*,\inf}^M(T)},
\]
whenever the denominator is nonzero. If it is zero, then let corresponding $\tau$ be $0$ if the numerator also equals 0 and $\infty$ if the numerator is positive. Alternatively we can use definitions without case distinction:
\[
	\tau_0^M(T) = \inf \left\{ t \ge 0 : \rho_{0,\sup}^M(T) \le e^t \rho_{0,\inf}^M(T) \right\}, \quad \tau_*^M(T) = \inf \left\{ t \ge 0 : \rho_{*,\sup}^M(T) \le e^t \rho_{*,\inf}^M(T) \right\}
\]
with the convention $\inf \emptyset = \infty$.

We further obtain a bound for $\tau_0^1$ in terms of $\deg(T)$. If $\deg(T) = 0$, that is $T = \bullet$ is the tree with a single node, then it is easy to see from the definitions that $\tau_0^1(T) = 0$.
If $\deg(T) \ge 1$, then recursion \eqref{eq:rho0rec} and the inequalities $\rho_0 \le 1/2$ and $\rho_* \le 1$ imply
\begin{equation}\label{eq:rho0lower}
	\rho_0(T)\ge \frac{1 - 2^{-\deg(T)}}{1 - 2^{-\deg(T)} + 2^{\deg(T)}} \ge \frac{1/2}{1/2 + 2^{\deg(T)}}
\end{equation}
whence $\tau_0^1(T) \le \log \left( 1/2 + 2^{\deg(T)} \right) \ll \deg(T)$.
We conclude that for all trees $T$, we have
\begin{equation}\label{eq:tau0w}
	\tau_0^1(T) \ll \deg(T) = w_1(T).
\end{equation}
Unfortunately, it is not possible to bound $\tau_*^1(T)$, since whenever $T$ has a leaf of depth one, it cannot be dominated without including the root and therefore $d_*(T) = 0$. On the other hand, any other tree has $d_*(T) > 0$ whence $\tau_*^1(T) = \infty$. We can, however, bound $\tau_*^2(T)$. If $T$ has a leaf at depth one, then $d_{*,\sup}^2(T) = d_{*\inf}^2(T) = 0$ and hence $\tau_*^2(T) = 0$, so further assuming that no node of depth one is a leaf, we obtain from \eqref{eq:rhostarrec} and \eqref{eq:rho0lower} (using $\rho_{*,\sup}^2(T) \le 1$) that
\begin{equation}\label{eq:taustarw}
	\tau_*^2(T) = \log \frac{\rho_{*,\sup}^2(T)}{\rho_{*,\inf}^2(T)} \le \log \frac{1 + 2^{\deg(T)}}{\prod_i 1/(1 + 2^{\deg(T_i)+1})} \ll \deg(T) + \sum_i \deg(T_i) = w_1(T) + w_2(T).
\end{equation}

Our goal is now to show the following. 
\begin{lemma}\label{lem:etarec}
	For some constants $a > 1$ and $0 < c < 1$, the functional $\eta^M(T) := a\tau_0^M(T) + \tau_*^M(T)$ satisfies, for $M \ge 3$,
	\begin{equation*}
		\eta^M(T) \le 
		\begin{cases}
			c \sum_{v \in V_1(T)} \eta^{M-1}(T_v), \quad &\deg(T) \ge 2, \\
			c^2 \sum_{v \in V_2(T)} \eta^{M-2}(T_v), \quad &\deg(T) \le 1,
		\end{cases}
	\end{equation*}
	where $V_i(T)$ are the nodes of $T$ at depth (i.e., distance from the root) $i$.
\end{lemma}
Before we prove Lemma~\ref{lem:etarec}, let us show how it implies that the remaining conditions of Theorem~\ref{thm:normality} are satisfied.

By applying Lemma~\ref{lem:etarec} recursively to $T_v$ for nodes $v$ at depth $M-4$ or less, we are eventually left with a linear combination over nodes at depths $M-3$ and $M-2$, and hence we obtain a bound 
\begin{equation}
	\eta^M(T) \le c^{M-3} \sum_{v \in V_{M-3}(T)} \eta^{3}(T_v) + c^{M-2} \sum_{v \in V_{M-2}(T)} \eta^{2}(T).
\end{equation}
In view of \eqref{eq:tau0w} and \eqref{eq:taustarw} we can bound $\eta^3(T_v)$ and $\eta^2(T_v)$ by the number of nodes of $T_v$ at depth $1$ or $2$ (possibly multiplied by a constant), so that we obtain 
\begin{equation}\label{eq:etafinal}
	\eta^M(T) \ll c^M (w_{M-2}(T) + w_{M-1}(T) + w_M(T))
\end{equation}
and since $\tau_0^M, \tau_*^M \le \eta^M$, using \eqref{eq:dom_frho} we obtain
\begin{equation}
	|f(T) - f(T^M)| \le \log \frac{
		\rho_{0,\inf}^M(T)e^{\tau_0^M(T)} + \rho_{*,\inf}^M(T)e^{\tau_*^M(T)}
	}
	{
		\rho_{0,\inf}^M(T) + \rho_{*,\inf}^M(T) 
	}
	\le \log e^{\eta^M(T)} = \eta^M(T),
\end{equation}
which together with \eqref{eq:etafinal} and the same bounds for the expectations of $w_k$ that have been used in the previous examples shows that conditions \eqref{eq:thm1-2} and \eqref{eq:thm1-3} are satisfied for a suitable choice of $p_M$ and $M_n$ with $p_M \ll Mc^M$ and $M_n \ll \log n$, so that \eqref{eq:thm1-4} is satisfied.
\begin{proof}[Proof of Lemma~\ref{lem:etarec}]
Before considering the two cases $\deg(T) \ge 2$ and $\deg(T) \le 1$, let us bound the ratio of $A_{\sup} := \prod_i \left( 1 + \rho_{*,\sup}^m(T_i) \right)$ and $A_{\inf} := \prod_i \left( 1 + \rho_{*,\inf}^m(T_i) \right)$, $m \ge 1$.

Denote $\rho_i = \rho_{*,\sup}^{m}(T_i)$, $\tau_i = \tau_*^{m}(T_i)$, and $\Sigma_* = \sum_i \tau_*^{m}(T_i)$. Since $\rho_{*,\sup}^{m}(T_i) = \rho_{*,\inf}^{m}(T_i)$ when $\deg(T_i) = 0$, we can further assume that $\deg(T_i) \ge 1$ for all $i$. Hence, in view of recursion \eqref{eq:rhostarrec} and the bound $\rho_0(T) \le 1/2$ we can assume $\rho_i \le 1/2$. By definition, $\rho_{*,\inf}^m(T_i) = \rho_i e^{-\tau_i}$, which gives
\begin{equation}\label{eq:ASigma}
		\log \frac{A_{\sup}}{A_{\inf}} = \sum_i \left(- \log \left( \frac{1  + \rho_ie^{-\tau_i}}{1 + \rho_i} \right) \right) \le \sum_i \frac{\rho_i}{1 + \rho_i}\tau_i  \le \frac{1}{3} \sum_i \tau_i = \frac{1}{3} \Sigma_*,
\end{equation}
	where the first inequality follows from Jensen's inequality applied to the convex function $x \mapsto -\log x$. 

\emph{First case: $\deg(T) \ge 2$.}
	Since the right-hand side of \eqref{eq:rho0rec} is a decreasing function of each $\rho_0(T_i)$ and each $\rho_*(T_i)$, we have
	\begin{equation*}
		\tau_0^M(T) \le \log \left( \frac{1 - \prod_i \rho_{0,\sup}^{M-1}(T_i) + \prod_i \left( 1 + \rho_{*,\sup}^{M-1}(T_i) \right)}{1 - \prod_i \rho_{0,\inf}^{M-1}(T_i) + \prod_i \left( 1 + \rho_{*,\inf}^{M-1}(T_i) \right)} \cdot \frac{1 - \prod_i \rho_{0,\inf}^{M-1}(T_i)}{1 - \prod_i \rho_{0,\sup}^{M-1}(T_i)} \right)	
	\end{equation*}
	Writing $\Pi = \prod_i \rho_{0,\sup}^{M-1}(T_i)$ and $\Sigma_0 = \sum_i \tau_0^{M-1}(T_i)$ we have
	\begin{equation}\label{eq:tau0logs}
		\tau_0^M(T) \le \log \frac{\prod_i \left( 1 + \rho_{*,\sup}^{M-1}(T_i) \right)}{\prod_i \left( 1 + \rho_{*,\inf}^{M-1}(T_i) \right)} + \log \left( \frac{1 - \Pi e^{-\Sigma_0}}{1 - \Pi} \right).	
	\end{equation}
	The first term is at most $\Sigma_*/3$ by \eqref{eq:ASigma} with $m = M-1$.
	Turning to the second term in \eqref{eq:tau0logs}, and using $\Pi \le 1/4$ (since $\deg(T) \ge 2$), we get
	\begin{equation}\label{eq:logSigma0}
		\log \left( \frac{1 - \Pi e^{-\Sigma_0}}{1 - \Pi} \right) \le \log \left( \frac{1 - \frac{1}{4}e^{-\Sigma_0}}{1 - \frac{1}{4}} \right) \le \frac{1}{3}\Sigma_0, 
	\end{equation}
	where the inequality follows from Jensen's inequality: $1 \le \frac{3}{4}e^{\Sigma_0/3} +\frac{1}{4}e^{-\Sigma_0}$.

	Combining \eqref{eq:tau0logs}, \eqref{eq:ASigma} with $m = M-1$, and \eqref{eq:logSigma0} we obtain
	\begin{equation}
		\tau_0^M(T) \le \frac{1}{3} \sum_i \tau_0^{M-1}(T_i) + \frac{1}{3} \sum_i \tau_*^{M-1}(T_i).
		\label{eq:tau0deg2}
	\end{equation}
We now proceed to $\tau_*^M(T)$.
	Since the right-hand side of \eqref{eq:rhostarrec} is increasing in each $\rho_0(T_i)$ and decreasing in each $\rho_*(T_i)$, we have
	\begin{equation*}
		\tau_*^M(T) \le \log \left( \frac{\prod_i \rho_{0,\sup}^{M-1}(T_i)}{\prod_i \rho_{0,\inf}^{M-1}(T_i)} \cdot \frac{1 - \prod_i \rho_{0,\inf}^{M-1}(T_i) + \prod_i \left( 1 + \rho_{*,\sup}^{M-1}(T_i) \right)}{1 - \prod_i \rho_{0,\sup}^{M-1}(T_i) + \prod_i \left( 1 + \rho_{*,\inf}^{M-1}(T_i) \right)} \right)	
	\end{equation*}
	Keeping the notations $\Pi, A_{\sup}, A_{\inf}, \Sigma_0$ as above and noting that \eqref{eq:ASigma} implies $A_{\sup} = A_{\inf}e^{\Sigma_*/3}$, we have
	\begin{multline}\label{eq:taustarlogs}
		\tau_*^M(T) \le \log \left( \frac{\Pi}{\Pi e^{-\Sigma_0}} \right) + \log \left( \frac{1 - \Pi e^{-\Sigma_0} + A_{\sup}}{1 - \Pi + A_{\inf}} \right) 
		= \Sigma_0 + \log \left( \frac{1 - \Pi e^{-\Sigma_0} + A_{\inf}e^{\Sigma_*/3}}{1 - \Pi + A_{\inf}} \right) \\
		\le \Sigma_0 + \log \left( \frac{(1 - \Pi e^{-\Sigma_0} + A_{\inf})e^{\Sigma_*/3}}{1 - \Pi + A_{\inf}} \right) = \Sigma_0 + \frac{1}{3}\Sigma_* + \log \left( \frac{1 - \Pi e^{-\Sigma_0} + A_{\inf}}{1 - \Pi + A_{\inf}} \right).
	\end{multline}
	Further shortening $A = A_{\inf}$, we note that the argument of the last logarithm is decreasing in $A \ge 1$ and increasing in $\Pi \le 1/4$, hence we obtain
	\begin{equation*}
		\frac{1 - \Pi e^{-\Sigma_0} + A}{1 - \Pi + A} \le \frac{2 - \frac{1}{4}e^{-\Sigma_0}}{2 - \frac{1}{4}} = \frac{8 - e^{-\Sigma_0}}{7} \le e^{\Sigma_0/7} 
	\end{equation*}
	where the last inequality follows from Jensen's inequality: $\frac{7}{8}e^{\Sigma_0/7} + \frac{1}{8}e^{-\Sigma_0} \ge 1$.
	Putting the last estimate into \eqref{eq:taustarlogs}, we conclude
	\begin{equation*}
		\tau_*^M(T) \le \frac{8}{7} \Sigma_0 + \frac{1}{3}\Sigma_*.
	\end{equation*}
	Now combining this inequality with \eqref{eq:tau0deg2} and choosing $a = 13/7$, say, we obtain
	\begin{equation}\label{eq:etadeg2}
		\eta^M(T) \le \left(\frac{13}{7 \cdot 3}+ \frac{8}{7}\right)\Sigma_0 + \left(\frac{13}{7 \cdot 3} + \frac{1}{3}\right)\Sigma_* \le \frac{20\cdot 13}{21 \cdot 7}\Sigma_0 + \frac{20}{21}\Sigma_* = \frac{20}{21} \sum_i \eta^{M-1}(T_i).
	\end{equation}

	\emph{Second case: $\deg(T) \le 1$}. Since the case $\deg(T) = 0$ is trivial, let us further assume $\deg(T) = 1$. Let us write $T'$ for $T$ with the root removed and let us denote the branches of the root of $T'$ by $T_1, T_2, \dots$.
	Using \eqref{eq:rho0rec} and \eqref{eq:rhostarrec}, after some straightforward simplifications we obtain
	\begin{equation*}
		\rho_0(T) = \frac{1 - \rho_0(T')}{2 - \rho_0(T') + \rho_*(T')} = \frac{\prod_i \left( 1 + \rho_*(T_i) \right)}{1 + 2\prod_i \left( 1 + \rho_*(T_i) \right)}
	\end{equation*}
	and
	\begin{equation*}
		\rho_*(T) = \frac{\rho_0(T')}{2 - \rho_0(T') + \rho_*(T')} = \frac{1 - \prod_i \rho_0(T_i)}{1 + 2\prod_i \left( 1 + \rho_*(T_i) \right)}.
	\end{equation*}
	By obvious monotonicity properties, we obtain that
	\begin{equation}
		\tau_0^M(T) \le \log \left( \frac{\prod_i (1 + \rho_{*,\sup}^{M-2}(T_i))}{\prod_i (1 + \rho_{*,\inf}^{M-2}(T_i))} \cdot \frac{1 + 2\prod_i (1 + \rho_{*,\inf}^{M-2}(T_i))}{1 + 2\prod_i (1 + \rho_{*,\sup}^{M-2}(T_i))} \right) \le \log \left( \frac{\prod_i (1 + \rho_{*,\sup}^{M-2}(T_i))}{\prod_i (1 + \rho_{*,\inf}^{M-2}(T_i))}  \right)
	\end{equation}
	Writing $A = \prod_i (1 + \rho_{*,\inf}^{M-2}(T_i))$, $\Sigma_* = \sum_i \tau_*^{M-2}(T_i)$ and using \eqref{eq:ASigma} with $m = M-2$ we get
	\begin{equation}
		\label{eq:tau0deg1}
		\tau_0^M(T) \le \log \left( Ae^{\Sigma_*/3}/A \right) = \frac{1}{3}\Sigma_*. 
	\end{equation}
	On the other hand, writing $\Pi = \prod_i \rho_{0,\sup}^{M-2}(T_i) \le 1/2$ and using \eqref{eq:ASigma} with $m = M-2$, we get
	\begin{multline}\label{eq:taustarloglog}
		\tau_*^M(T) \le \log \left( \frac{1 - \prod_i \rho_{0,\inf}^{M-2}(T_i)}{1 - \prod_i \rho_{0,\sup}^{M-2}(T_i)} \cdot \frac{1 + 2\prod_i (1 + \rho_{*,\sup}^{M-2}(T_i))}{1 + 2\prod_i (1 + \rho_{*,\inf}^{M-2}(T_i))} \right) \\
		\le \log \left( \frac{1 - \Pi e^{-\Sigma_0}}{1 - \Pi} \right) + \log \left( \frac{\prod_i (1 + \rho_{*,\sup}^{M-2}(T_i))}{\prod_i (1 + \rho_{*,\inf}^{M-2}(T_i))}  \right) \le \log \left( 2 - e^{-\Sigma_0} \right) + \frac{1}{3}\Sigma_*.
	\end{multline}
	The basic inequality $e^{\Sigma_0} + e^{-\Sigma_0} \ge 2$ implies $\log(2- e^{-\Sigma_0}) \le \Sigma_0$, which together with \eqref{eq:taustarloglog} implies
	\begin{equation}
		\label{eq:taustardeg1}
		\tau_*^M(T) \le \Sigma_0 + \frac{1}{3}\Sigma_*.
	\end{equation}
	This, combined with \eqref{eq:tau0deg1}, implies (recall we chose $a = 13/7$)
	\begin{equation}
		\eta(T) = \frac{13}{7}\tau_0^M(T) + \tau_*^M(T) \le \Sigma_0 + \frac{20}{21}\Sigma_* \le \frac{20}{21}\sum_i \eta^{M-2}(T_i).
		\label{eq:etadeg1}
	\end{equation}
	Combining \eqref{eq:etadeg1} with \eqref{eq:etadeg2} completes the proof of Lemma~\ref{lem:etarec} with $c = \sqrt{20/21}$.
\end{proof}

Now that we know that all conditions of Theorem~\ref{thm:normality} are indeed satisfied (with $\alpha = 1$), we have the following theorem:

\begin{theorem}
Let $\tc_n$ be a conditioned Galton-Watson tree of order $n$ with offspring distribution $\xi$, where $\xi$ satisfies $\E \xi = 1$ and $0<\sigma^2:=\mathrm{Var} \xi <\infty$ as well as $\E \xi^3 < \infty$. There exist constants $\mu > 0$ and $\gamma > 0$ (both depending on $\xi$) such that
$$\frac{\log d (\tc_n) - n\mu}{\sqrt n} \dto \nc(0,\gamma^2)$$
as $n \to \infty$.
\end{theorem}

\subsection{Tree reductions}
In a recent paper \cite{HHKP17},  Hackl, Heuberger, Kropf, and Prodinger studied various natural tree reduction processes based on repeatedly deleting parts of the tree by certain operations until the root is isolated. Before we describe these processes, we first need to define a few terms. A leaf is called an \textit{old leaf} if it is the leftmost child of its parent node (we assume that the root of the tree cannot be an old leaf). A maximal fringe subtree rooted at a node other that the root with the property that each of its nodes has outdegree 0 or 1 is simply referred to as a \textit{path}. An \textit{old path} is maximal fringe subtree rooted at a node other that the root with the property that each of its nodes is the leftmost child of its parent. In \cite{HHKP17}, the tree is reduced until the root is isolated by repeating one of the following operations:     
\begin{enumerate}
\item[(a)] Leaf-reduction: all leaves are deleted in each round,
\item[(b)] Old leaf-reduction: all old leaves are deleted in each round,
\item[(c)] Path-reduction: all paths are deleted in each round,
\item[(d)] Old path-reduction: all old paths are deleted in each round.
\end{enumerate}
In any of the above operations when a node is deleted, the incident edges are also deleted from the tree. An example is given in Figure~\ref{fig:reduction}, where, for each of the four operations, the parts of the tree that are about to be deleted in the next round of reduction are dashed.

\begin{figure}[h]
\begin{minipage}{.2\textwidth}
\begin{tikzpicture}[main_node/.style={circle,draw,minimum size=0.4em,inner sep=2pt]},scale=0.5]
    \node[main_node, draw=none] (1) at (-3,0) {};
    \node[main_node, draw=none,label=below:(a) Leaves  \, \, ] (1) at (0,-2) {};
    \node[main_node,fill] (0) at (0, 3){};
    \node[main_node,fill] (1) at (1,2) {};
    \node[main_node,fill] (11) at (0,1) {};
    \node[main_node, dashed] (12) at (0,0) {};
    \node[main_node, dashed] (13) at (1,1) {};
    \node[main_node,fill] (2) at (-1, 2){};
    \node[main_node, fill] (3) at (-2, 1){};
    \node[main_node, fill] (4) at (-3, 0){};
    \node[main_node,dashed] (6) at (-3, -1){};
    \node[main_node, fill] (14) at (2, 1){};
    \node[main_node, fill] (15) at (2, 0){};
    \node[main_node, dashed] (16) at (2, -1){};
    \node[main_node,dashed] (5) at (-1, 0){};
    \draw[dashed] (4)--(6);
    \draw (4)--(3)--(2)--(0)--(1);
    \draw[dashed] (3)--(5);
    \draw (1)--(14);
    \draw (14)--(15);
    \draw[dashed](15)--(16);
    \draw[dashed] (13)--(1);    
    \draw (1)--(11);
    \draw[dashed] (11)--(12);
\end{tikzpicture}
\end{minipage}
\begin{minipage}{.2\textwidth}
\begin{tikzpicture}[main_node/.style={circle,draw,minimum size=0.4em,inner sep=2pt]},scale=0.5]
    \node[main_node, draw=none] (1) at (-3,0) {};
    \node[main_node, draw=none,label=below:(b) Old leaves \, \,] (1) at (0,-2) {};
    \node[main_node,fill] (0) at (0, 3){};
    \node[main_node,fill] (1) at (1,2) {};
    \node[main_node,fill] (11) at (0,1) {};
    \node[main_node, dashed] (12) at (0,0) {};
    \node[main_node,fill] (13) at (1,1) {};
    \node[main_node,fill] (2) at (-1, 2){};
    \node[main_node, fill] (3) at (-2, 1){};
    \node[main_node, fill] (4) at (-3, 0){};
    \node[main_node,dashed] (6) at (-3, -1){};
    \node[main_node, fill] (14) at (2, 1){};
    \node[main_node, fill] (15) at (2, 0){};
    \node[main_node, dashed] (16) at (2, -1){};
    \node[main_node, fill] (5) at (-1, 0){};
    \draw[dashed] (4)--(6);
    \draw (4)--(3)--(2)--(0)--(1);
    \draw (3)--(5);
    \draw (1)--(14);
    \draw (14)--(15);
    \draw[dashed](15)--(16);
    \draw (13)--(1);    
    \draw (1)--(11);
    \draw[dashed] (11)--(12);
\end{tikzpicture}
\end{minipage}
\begin{minipage}{.2\textwidth}
\begin{tikzpicture}[main_node/.style={circle,draw,minimum size=0.4em,inner sep=2pt]},scale=0.5]
    \node[main_node, draw=none] (1) at (-3,0) {};
    \node[main_node, draw=none,label=below:(c) Paths  \, \,] (1) at (0,-2) {};
    \node[main_node,fill] (0) at (0, 3){};
    \node[main_node,fill] (1) at (1,2) {};
    \node[main_node,dashed] (11) at (0,1) {};
    \node[main_node, dashed] (12) at (0,0) {};
    \node[main_node, dashed] (13) at (1,1) {};
    \node[main_node,fill] (2) at (-1, 2){};
    \node[main_node, fill] (3) at (-2, 1){};
    \node[main_node, dashed] (4) at (-3, 0){};
    \node[main_node,dashed] (6) at (-3, -1){};
    \node[main_node, dashed] (14) at (2, 1){};
    \node[main_node, dashed] (15) at (2, 0){};
    \node[main_node, dashed] (16) at (2, -1){};
    \node[main_node,dashed] (5) at (-1, 0){};
    \draw[dashed] (4)--(6);
    \draw[dashed] (4)--(3);
    \draw (3)--(2)--(0)--(1);
    \draw[dashed] (3)--(5);
    \draw[dashed] (1)--(14)--(15)--(16);
    \draw[dashed] (13)--(1)--(11)--(12);    
\end{tikzpicture}
\end{minipage}
\begin{minipage}{.2\textwidth}
\begin{tikzpicture}[main_node/.style={circle,draw,minimum size=0.4em,inner sep=2pt]},scale=0.5] 
    \node[main_node, draw=none] (1) at (-3,0) {};
    \node[main_node, draw=none,label=below: (d) Old paths  \, \, ] (1) at (0,-2) {};
    \node[main_node,fill] (0) at (0, 3){};
    \node[main_node,fill] (1) at (1,2) {};
    \node[main_node,dashed] (11) at (0,1) {};
    \node[main_node, dashed] (12) at (0,0) {};
    \node[main_node, fill] (13) at (1,1) {};
    \node[main_node,fill] (2) at (-1, 2){};
    \node[main_node, fill] (3) at (-2, 1){};
    \node[main_node, dashed] (4) at (-3, 0){};
    \node[main_node,dashed] (6) at (-3, -1){};
    \node[main_node, fill] (14) at (2, 1){};
    \node[main_node, dashed] (15) at (2, 0){};
    \node[main_node, dashed] (16) at (2, -1){};
    \node[main_node,fill] (5) at (-1, 0){};
    \draw[dashed] (4)--(6);
    \draw[dashed] (4)--(3);
    \draw (3)--(2)--(0)--(1);
    \draw (3)--(5);
    \draw (1)--(14);
    \draw[dashed] (14)--(15)--(16);
    \draw (13)--(1);
    \draw[dashed] (1)--(11)--(12);
\end{tikzpicture}
\end{minipage}

\caption{Tree reduction operations}\label{fig:reduction}
\end{figure}
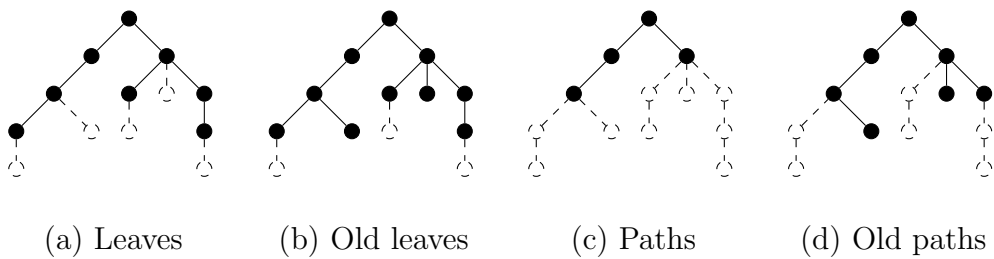
For a given positive integer $r$, and for a tree $T$, let $X_{r}(T)$  be the number of nodes in the reduced tree after the first $r$ steps of one of the above reductions. The authors of \cite{HHKP17} proved asymptotic estimates for the mean and variance as well as a central limit theorem for $X_{r}(\tc_n)$ for the uniform random  plane (=ordered) tree on $n$ nodes in the cases (a)--(c). For the case (d), they gave asymptotic estimates for mean and variance of $X_{r}(\tc_n)$, but left the central limit theorem as an open problem. We show that asymptotic normality of the functional $X_{r}(\tc_n)$ can also be derived from Theorem~\ref{thm:normality} for conditioned Galton-Watson trees, of which the uniform random plane tree is a special case. We let   
$$
F_r(T)=|T|-X_r(T),
$$   
which corresponds to the number of nodes of $T$ other than the root that are deleted after $r$ steps according to one of the reductions above. The functional $F_r$ is  additive with toll function $f_r$, where
$$
f_r(T)=\sum_{j}\eta_T(T_j)
$$ 
and the sum is over all branches $T_j$, with
$$
\eta_T(T_j)=
\begin{cases}
& 1 \, \text{ if the root of $T_j$ is deleted within the first $r$ steps,}\\
& 0 \, \text{ otherwise.}
\end{cases}
$$
We can immediately see that 
$$
0 \leq f_r(T)\leq \deg (T).
$$
Hence, we can take $\alpha=1$ which implies that our offspring distribution $\xi$ is required to have a finite third moment.  In fact, $f_r(T)$ is upper bounded by $r$ for the old leaf- and old path-reductions since the children of the root can only be deleted one at a time. However, the bound is clearly sharp for the leaf- and path-reductions. 

Next, we show that  $f_r$ sastisfies the remaining conditions (i.e., \eqref{eq:thm1-2}, \eqref{eq:thm1-3}, \eqref{eq:thm1-4}) of Theorem~\ref{thm:normality} in all four cases. For a rooted tree $T$, we denote by $T^*$ the planted tree where the root of $T$ is connected to a new node, which becomes the root of $T^*$.   Let 
$
\kappa=\min \{k\geq 2 \, : \, \prob{\xi=k}>0\}
$
(this value must exist under our assumptions on $\xi$), and let $T_0$ be the complete $\kappa$-ary tree  of depth $r$. It is easy to verify that $X_r(T_0^*)\neq 1$, i.e.~$T_0^*$ is not reduced to the root in $r$ steps, and 
\begin{equation}\label{eq:red1}
\prob{\tc= T_0}>0. 
\end{equation}
For each positive integer $M$, let $\mathcal{B}_{M}$ be the set of all trees $T$ (not necessarily finite) of height at least $M-1$ such that $X_r((T^{(M-1)})^*)=1$ (i.e.~the tree $T^{(M-1)}$ vanishes after the first $r$ steps of the reduction). It is important to notice here that a rooted tree $T$ is not reduced to a single node after the first $r$ steps of the reduction if  the fixed tree $T_0$ appears as a \textit{subtree} of $T$ (here, by \textit{subtree}, we mean a subtree of the form $T_{v}^{(r)}$ and some node $v$ of $T$). This observation is key in the proof of the next lemma.
\begin{lemma}\label{lem:red} There is a positive constant $c<1$ that depends only on $\xi$ and $r$, such that  
	$$
	\prob{\tc \in \mathcal{B}_M} \ll c^{M} \text{ and }\, \prob{\htc \in \mathcal{B}_M} \ll c^{M}. 
	$$
\end{lemma}

\begin{proof}
Without loss of generality we assume $M \ge r$.
	We start with the first estimate. We notice that for $\tc$ to be in $\mathcal{B}_M$, $\tc^{(r)}$ must not be equal to $T_0$. Moreover, the assumption $M - 1 \ge r$ implies $w_r(\tc) \ge 1$. So  
	\begin{align*}
		\prob{\tc \in \mathcal{B}_M}=\sum_{T\neq T_0 : w_r(T)\ge 1}\prob{\tc^{(r)}=T}\prob{\tc \in \mathcal{B}_M \, |\, \tc^{(r)}=T}.
	\end{align*}
	Conditioning on the event $\{\tc^{(r)}=T \}$, the rest of $\tc$ is a forest consisting of $w_r(T)$ independent copies of $\tc$. If $\tc \in \mathcal B_M$, then all of them must belong to $\mathcal{B}_{M-r}$, hence we obtain 
	\begin{align*}
		\prob{\tc \in \mathcal{B}_M}
		& \leq \sum_{T\neq T_0 : w_r(T) \ge 1}\prob{\tc^{(r)}=T} \prob{\tc\in \mathcal{B}_{M-r}}^{w_r(T)}\\
		& \leq \sum_{T\neq T_0}\prob{\tc^{(r)}=T} \prob{\tc\in \mathcal{B}_{M-r}}\\
		& = \prob{\tc\in \mathcal{B}_{M-r}} \prob{\tc^{(r)} \neq T_0}\\
		& \leq  \prob{\tc\in \mathcal{B}_{M-r}} \prob{\tc \neq T_0} =: q\prob{\tc \in \mathcal{B}_{M-r}},
	\end{align*}
	where $q < 1$ by \eqref{eq:red1}.
Iterating this inequality gives
	\begin{equation}\label{eq:red3}
	\prob{\tc \in \mathcal{B}_M}\leq q^{\lfloor M/r\rfloor} \ll c^{M},
	\end{equation}
	where $c :=q^{1/r} < 1$, proving the first estimate.

	For the second estimate, we also begin in a similar fashion, i.e.~we have
	\begin{align*}
		\prob{\htc \in \mathcal{B}_M}=\sum_{T\neq T_0 : w_r(T) \ge 1}\prob{\htc^{(r)}=T}\prob{\htc \in \mathcal{B}_M \, |\, \htc^{(r)}=T}.
	\end{align*}
	Here, when conditioning on the event $\{\htc^{(r)}=T \}$, the rest of $\htc$ is a forest consisting of $w_r(T)-1$ independent copies of $\tc$ and an independent copy of $\htc$. Thus,
	\begin{multline*}
	\prob{\htc \in \mathcal{B}_M}  
	\leq\sum_{T\neq T_0 : w_r(T) \ge 1}\prob{\htc^{(r)}=T} \prob{\tc\in \mathcal{B}_{M-r}}^{w_r(T)-1} \prob{\htc\in \mathcal{B}_{M-r}} \\
	\leq\sum_{T\neq T_0}\prob{\htc^{(r)}=T}\prob{\htc\in \mathcal{B}_{M-r}} \leq \prob{\tc \neq T_0} \prob{\htc\in \mathcal{B}_{M-r}}
	\end{multline*}
and similarly as in the previous case, using \eqref{eq:red1} we obtain
	$$
	\prob{\htc \in \mathcal{B}_M}\ll c^{M},
	$$ 
	which completes the proof.
\end{proof}

For a finite tree $T$, the only possibility for which $f_r(T\po)\neq f_r(T)$ is when there is a root branch $T_j$ of $T$ such that  $T_j^{(M-1)}$  vanishes after the first $r$ steps of the reduction of $T^{(M)}$, but $T_j$ does not vanish after the first $r$ steps of the reduction of $T$. This means that if $f_r(T\po)\neq f_r(T)$, then $T$ must have a branch in $\mathcal{B}_M$. Therefore, we have
\begin{align*}
	\prob{f_r(\tc\po)\neq f_r(\tc)}\leq \sum_{k=1}^{\infty}k\prob{\xi=k} \prob{\tc\in \mathcal{B}_M } \ll c^M.
\end{align*}
The estimate on the right follows from Lemma~\ref{lem:red}. As an immediate consequence of this, we have 
$$
\prob{f_r(\tc_n\po)\neq f_r(\tc_n)}\leq \frac{\prob{f_r(\tc\po)\neq f_r(\tc)}}{\prob{|\tc|=n}}\ll n^{3/2} \, c^M,
$$
using a well-known asymptotic $\prob{|\tc| = n} = \Theta(n^{-3/2})$, see \cite[(4.13)]{J16}.
Hence,
\begin{equation}\label{eq:red4}
\E |f_r(\tc_n\po)-f_r(\tc_n)|\ll n^{3/2}c^{M}\, \max_{|T|=n}|f_r(T\po)-f_r(T)|\ll n^{5/2}c^M.
\end{equation}

Let us denote by $\mathcal{E}_M$ the event $\bigcup_{N>M}\{f_r(\htc\po)\neq\E( f_r(\htc^{(N)})\, | \, \htc\po)\}$. Then, for any $N\geq M$, we have 
$$
\Big|f_r(\htc\po)-\E( f_r(\htc^{(N)})\, | \, \htc\po)\Big|\ll \deg(\htc\po) \I_{\mathcal{E}_M}.
$$
For $\htc$ to be in $\mathcal{E}_M$, $\htc$ must have a root branch in $\mathcal{B}_M$. Therefore, 
\begin{multline}
	\E \Big|f_r(\htc\po)-\E( f_r(\htc^{(N)})\, | \, \htc\po)\Big|\\ 
	\ll \sum_{k=1}^{\infty}k\prob{\deg(\htc)=k}\left((k-1)\prob{\tc\in \mathcal{B}_M } +\prob{\htc\in \mathcal{B}_M } \right).
\end{multline}
In view of Lemma~\ref{lem:red}, we have 
\begin{equation}\label{eq:red5}
\E \Big|f_r(\htc\po)-\E( f_r(\htc^{(N)})\, |\, \htc\po)\Big| \ll c^M\sum_{k=1}^{\infty}k^2\prob{\deg(\htc)=k} \ll c^M, 
\end{equation}
since $\E(\deg(\htc)^2) <\infty$ if $\E \xi^3<\infty$, see \eqref{eq:prem3}. The estimates \eqref{eq:red4} and \eqref{eq:red5} confirm that $f_r$ satisfies all conditions of Theorem~\ref{thm:normality} for a suitable choice of $p_M$ and $M_n$ with $p_M  \ll c_2^M$ for some $c_2 \in (c,\, 1)$  and $M_n$ satisfying $M_n \ll \log n$.

\begin{theorem}
Let $\tc_n$ be a conditioned Galton-Watson tree of order $n$ with offspring distribution $\xi$, where $\xi$ satisfies $\E \xi = 1$ and $0<\sigma^2:=\mathrm{Var} \xi <\infty$ as well as $\E \xi^3 < \infty$. Consider one of the four reduction procedures described at the beginning of this subsection, and fix a positive integer $r$, so that $F_r(T)$ denotes the number of deleted nodes after $r$ steps of the procedure. There exist constants $\mu > 0$ and $\gamma \geq 0$ (depending on $\xi$, the specific procedure and the value of $r$) such that
$$\frac{F_r (\tc_n) - n\mu}{\sqrt n} \dto \nc(0,\gamma^2)$$
as $n \to \infty$.
\end{theorem}

We remark that there are some (trivial) degenerate cases for this example: consider for instance the leaf reduction procedure applied to binary trees. The number of nodes removed in the first step is deterministic since the number of leaves in binary trees is.

\bibliographystyle{plain}

\end{document}